\documentclass[11pt,openany]{article}
\usepackage[utf8]{inputenc}
\usepackage[english]{babel}
\usepackage[margin=0.75in]{geometry}
\usepackage{color}
\usepackage{enumerate}
\usepackage{amsmath, amsthm, amssymb, amsfonts}
\usepackage{mathtools}
\usepackage[all,cmtip]{xy}
\usepackage{todonotes}
\usepackage{hyperref}
\usepackage{todonotes}
\allowdisplaybreaks
\newtheorem{theorem}{Theorem}[section]
\newtheorem{corollary}[theorem]{Corollary}
\newtheorem{lemma}[theorem]{Lemma}
\newtheorem{proposition}[theorem]{Proposition}
\theoremstyle{definition}
\newtheorem{definition}[theorem]{Definition}

\newtheorem{calc}{Calculation}
\newtheorem{calc'}{Calculation}
\newtheorem{calc1}{Calculation}

\newtheorem{case'}{Case}
\newtheorem{case''}{Case}

\newtheorem{claim*}{Claim}

\newtheorem*{lma}{Lemma}
\newtheorem*{prop}{Proposition}
\newtheorem*{note}{Note}
\newtheorem{remark}{Remark}
\newcommand{\C}{\mathbb{C}}

\newcommand{\Z}{\mathbb{Z}}
\newcommand{\F}{\mathbb{F}}

\newcommand{\Up}{\operatorname{Up}}
\newcommand{\ed}{\operatorname{ed}}
\newcommand{\syl}{\operatorname{Syl}}
\newcommand{\trdeg}{\operatorname{trdeg}}
\newcommand{\mbf}{\mathbf}
\newcommand{\sym}{Sym}
\newcommand{\asym}{Antisym}
\newcommand{\asymo}{Antisym_0}

\newcommand{\divides}{\bigm|}

\newenvironment{theorem*}[2][Theorem]{\begin{trivlist}
\item[\hskip \labelsep {\bfseries #1}\hskip \labelsep {\bfseries #2}]}{\end{trivlist}}
\newenvironment{lemma*}[2][Lemma]{\begin{trivlist}
\item[\hskip \labelsep {\bfseries #1}\hskip \labelsep {\bfseries #2}]}{\end{trivlist}}
\newenvironment{corollary*}[2][Corollary]{\begin{trivlist}
\item[\hskip \labelsep {\bfseries #1}\hskip \labelsep {\bfseries #2}]}{\end{trivlist}}
\begin{document}
\title{\texorpdfstring{The essential $p$-dimension of the split finite quasi-simple groups of classical Lie type}{The essential p-dimension of the split finite quasi-simple groups of classical Lie type}}
\author{Hannah Knight\thanks{This work was supported in part by NSF Grant Nos. DMS-1811846 and DMS-1944862.}\\ Department of Mathematics, UCI, Irvine, CA, USA\\ {\color{blue}\href{mailto:hknight1@uci.edu}{hknight1@uci.edu}}} 
\date{}

\maketitle

\begin{abstract}

In this paper, we compute the essential $p$-dimension of the split finite quasi-simple groups of classical Lie type at the defining prime, specifically the quasi-simple groups arising from the general linear and special linear groups, the symplectic groups, and the orthogonal groups.

\end{abstract}

\tableofcontents

\bigskip

\section{Introduction}

The goal of this paper is to compute the essential $p$-dimension at the defining prime for the finite groups of classical Lie type.
Fix a field $k$. The essential dimension of a finite group $G$, denoted $\ed_k(G)$, is the smallest number of algebraically independent parameters needed to define a Galois $G$-algebra over any field extension $F/k$ (or equivalently $G\text{-torsors over }\text{Spec}F$). In other words, the essential dimension of a finite group $G$ is the supremum taken over all field extensions $F/k$ of the smallest number of algebraically independent parameters needed to define a Galois $G$-algebra over $F$.  The essential $p$-dimension of a finite group, denoted $\ed_k(G,p)$, is similar except that before taking the supremum, we allow finite extensions of $F$ of degree prime to $p$ and take the minimum of the number of parameters needed.  In other words, the essential $p$-dimension of a finite group is the supremum taken over all fields $F/k$ of the smallest number of algebraically independent parameters needed to define a Galois $G$-algebra over a field extension of $F$ of degree prime to $p$.  See Section \ref{edbackground} for more formal definitions. See also \cite{BR} and \cite{KM} for more detailed discussions. For a discussion of some interesting applications of essential dimension and essential $p$-dimension, see \cite{Reich}.

What is the essential dimension of the finite simple groups? This question is quite difficult to answer. A few results for small groups (not necessarily simple) have been proven. For example, it is known that $\ed_k(S_5) = 2$, $\ed_k(S_6) = 3$ for $k$ of characteristic not $2$ \cite{BF}, and $\ed_k(A_7) = \ed_k(S_7) = 4$ in characteristic $0$ \cite{Dun}.  It is also known that for $k$ a field of characteristic $0$ containing all roots of unity, $\ed_k(G) = 1$ if and only if $G$ is isomorphic to a cyclic group $\Z/n\Z$ or a dihedral group $D_m$ where $m$ is odd (\cite{BR}, Theorem 6.2). Various bounds have also been proven. See \cite{BR}, \cite{Mer}, \cite{Reich},\cite{Mor}, among others. For a nice summary of the results known in 2010, see \cite{Reich}.

We can find a lower bound to this question by considering the corresponding question for essential $p$-dimension. In this paper, we prove:
\begin{theorem}\label{mainthm}

\begin{enumerate}[(1)] Let $p$ be a prime, $k$ a field with $\text{char } k \neq p$. Then

\item (Theorem \ref{edPSLnallp}, Bardestani-Mallahi-Karai-Salmasian $p \neq 2$ \cite{BMS}, K. $p = 2$)
$$\ed_k(PSL_n(\F_{p^r}),p) = \ed_k(GL_n(\F_{p^r}),p) = rp^{r(n-2)}.$$

\item (Theorem \ref{edPSp})
$$\ed_k(PSp(2n,p^r),p) = \ed_k(Sp(2n,p^r),p) = \begin{cases} rp^{r(n-1)}, &p \neq 2 \text{ or } n =2\\
r2^{r(n-1)-1}(2^{r(n-2)} + 1), &p = 2, n > 2 \end{cases}$$

\item (Theorem \ref{edO2n})
\begin{align*} 
 \ed_k(P\Omega^\epsilon(2m,p^r),p) = \ed_k(\Omega^\epsilon(2m,p^r),p) = \begin{cases}  2r, &2m = 4, \text{ any } p\\
 rp^{2r(m-2)}, &2m > 4, \text{ any } p\\
\end{cases}\end{align*}
Furthermore, $\ed_k(O^\epsilon(2m,2^r),2) = 1 + \ed_k(\Omega^\epsilon(2m,2^r),2),$ and  for $p \neq 2$,  $\ed_k(O^\epsilon(2m,p^r),p) = \ed_k(\Omega^\epsilon(2m,p^r),p)$.
\end{enumerate}
\end{theorem}

\begin{remark} In Theorem \ref{edPSp}, for $p = 2, n = 2, r = 1$, we have $PSp(4,2)' \cong A_6$, and so $\ed_k(PSp(4,2)',2) = \ed_k(A_6,2) = 2$. Except for $p = 2, n = 2, r = 1$, $PSp(2n,p^r) = PSp(2n,p^r)'$ is simple. The methods of this paper can recover the proof that $\ed_k(PSp(4,2),2) = \ed_k(S_6,2) = 3$ and that $\ed_k(PSp(4,2)',2) = \ed_k(A_6,2) = 2$, but for brevity, because these are known theorems, we will omit the proofs here. \end{remark}

\begin{remark} If $\text{char } k = p$, then $\ed_k(G,p) = 1$ unless $p \nmid |G|$, in which case $\ed_k(G,p) = 0$ \cite{RV}. So we have a full description of $\ed_k(G,p)$ for the finite simple groups of Lie type for any field $k$ and any $p$.  \end{remark}

\begin{remark} Dave Benson independently proved $\ed_\C(Sp(2n,p),p) = p^{n-1}$ for $p$ odd (\cite{BrFa}, Appendix A). \end{remark}

\begin{remark} The following results were known prior to this paper: 
\begin{enumerate}

\item $\ed_\C(PSL_n(\F_{p^r},p)) = \ed_\C(GL_n(\F_{p^r})) = rp^{r(n-2)}$ for $p \neq 2$ (\cite{BMS}, Theorems 1.1 and 1.2). 

\item Duncan and Reichstein calculated the essential $p$-dimension of the pseudo-reflection groups: For $G$ a pseudo-reflection group with $k[V]^G = k[f_1,\ldots,f_n]$, $d_i = \text{deg}(f_i),$ $\ed_k(G,p) = a(p) = |\{i : d_i \text{ is divisible by } p\}|$ (\cite{DR}, Theorem 1.1). These groups overlap with the groups above in a few small cases. See the appendix (\ref{Ap1}) for the overlapping cases .
%
%
%
%
%

\item Reichstein and Shukla calculated the essential $2$-dimension of double covers of the symmetric and alternating groups in characteristic $\neq 2$: Write $n = 2^{a_1} + \dots + 2^{a_s}$, where $a_1 > a_2 > \ldots > a_s \geq 0$. For $\tilde{S_n}$ a double cover of $S_n$, $\ed_k(\tilde{S_n},2) =2^{\lfloor(n-s)/2\rfloor}$, and for $\tilde{A_n}$ a double cover of $A_n$, $\ed_k(\tilde{A_n},2) = 2^{\lfloor(n-s-1)/2\rfloor}$ (\cite{RS}, Theorem 1.2). These groups overlap with the groups above in a few small cases: $\tilde{A_4} \cong SL_2(3)$, $\tilde{A_5} \cong SL_2(5)$, $\tilde{A_6} \cong SL_2(9)$, $\tilde{S_4^+} \cong GL_2(3)$. However, in each case the defining prime is not $2$, and as I only calculate the essential dimension at the defining prime, there is no overlap with the results of Reichstein and Shukla.
 
\end{enumerate} \end{remark}

\subsection{General Outline for Proofs}

\noindent The key tools in the proofs of Theorem \ref{mainthm} are the Karpenko-Merkurjev Theorem (Theorem \ref{KM4.1}), a lemma of Meyer and Reichstein (Lemma \ref{BMKS3.5}), and Wigner Mackey Theory.

\begin{theorem}\label{KM4.1} [Karpenko-Merkurjev \cite{KM}, Theorem 4.1] Let $G$ be a $p$-group, $k$ a field with $\text{char } k \neq p$ containing a primitive $p$-th root of unity. Then $\ed_k(G,p) = \ed_k(G)$ and $\ed_k(G,p)$ coincides with the least dimension of a faithful representation of $G$ over $k$.
\end{theorem}

\noindent The Karpenko-Merkurjev Theorem allows us to translate the question for $p$-groups formulated in terms of extensions and transcendence degree into a question of representation theory.  

\begin{lemma}\label{BMKS3.5}[\cite{MR}, Lemma 2.3] Let $k$ be a field with $\text{char } k \neq p$ containing $p$-th roots of unity. Let $H$ be a finite $p$-group and let $\rho$ be a faithful representation of $H$ of minimal dimension. Then $\rho$ decomposes as a direct sum of exactly $r = \text{rank}(Z(H))$ irreducible representations
$$\rho = \rho_1 \oplus \ldots \oplus \rho_r.$$ and 
if $\chi_i$ are the central characters of $\rho_i$, then $\{\chi_i|_{\Omega_1(Z(H))}\}$ is a basis for $\widehat{\Omega_1}(Z(H))$ over $k$.\\
($\Omega_1(Z(H))$ is defined to be the largest elementary abelian $p$-group contained in $Z(H)$; see Definition \ref{Def3.1}.) \end{lemma}

\noindent This lemma allows us to translate a question of analyzing faithful representations into a question of analyzing irreducible representations.  Our main tool for the case at hand is Wigner-Mackey Theory. This method from representation theory allows us to classify the irreducible representations for groups of the form $\Delta \rtimes L$ with $\Delta$ abelian. (See section \ref{WigMac}.)

By Lemma \ref{Lempsyl}, it suffices to consider the Sylow $p$-subgroups. By Corollary \ref{rootofunity}, we may assume that our field $k$ contains $p$-th roots of unity. Then by the Karpenko-Merkurjev Theorem, we need to find the minimal dimension of a faithful representation of the Sylow $p$-subgroups. Throughout this paper, we will use the notation $\syl_p(G)$ to denote the set of Sylow $p$-subgroups of $G$. Let $S \in \syl_p(G)$.   By Lemma \ref{BMKS3.5}, if the center of $S$ has rank $s$, a faithful representation $\rho$ of $S$ of minimal dimension decomposes as a direct sum
$$\rho = \rho_1 \oplus \ldots \oplus \rho_{s}$$
of exactly $s$ irreducibles, and if $\chi_i$ are the central characters of $\rho_i$, then $\{\chi_i|_{\Omega_1(Z(S))}\}$ is a basis for $\widehat{\Omega_1}(Z(S))$ (see Definition \ref{Def3.1}). 

Our proofs will follow the following steps:

\begin{itemize}

\item Step 1: Find the Sylow $p$-subgroups and their centers.

\item Step 2: Classify the irreducible representations of the Sylow $p$-subgroups using Wigner-Mackey theory.

\item Step 3: Construct upper and lower bounds using the classification in step 2.

\end{itemize}

\begin{remark} A more recent paper of Bardestani, Mallahi-Karai, and Salmasian \cite{BMS2} gives a way of calculating the essential dimension by means of commutator matrices. We do not pursue this method here.
\end{remark}

\noindent \textbf{Acknowledgements:} I would like to thank Jesse Wolfson for his kind mentorship and generous support. I am also grateful to Vladimir Chernousev, Alexander Duncan, Najmuddin Fakhruddin, Nate Harman, Michael Hehmann, Klaus Lux, Keivan Mallahi-Karai, Mark MacDonald, Anubhav Nanavaty, Zinovy Reichstein, Hadi Salmasian, Federico Scavia, Jean-Pierre Serre, Alex Sutherland, and  Burt Totaro for very helpful comments on a draft. I would also like to thank the anonymous referees for helpful comments which greatly streamlined the paper and improved the exposition.
 
\section{\texorpdfstring{Essential $p$-Dimension Background}{Essential p-Dimension Background}}\label{edbackground}
Fix a field $k$. Let $G$ be a finite group, $p$ a prime.

\begin{definition} Let $T: \text{Fields}/k \to \text{Sets}$ be a functor. Let $F/k$ be a field extension, and $t \in T(F)$. The\emph{ essential dimension of} $\mathit{t}$ is 
$$\ed_k(t) = \min_{F' \subset F \text{ s.t. } t \in Im(T(F') \to T(F))} \trdeg_k(F').$$\end{definition}

\begin{definition} Let $T: \text{Fields}/k \to \text{Sets}$ be a functor. The \emph{essential dimension of} $\mathit{T}$ is
$$\ed_{k}(T) = \sup_{t \in T(F), \text{ } F/k \in \text{Fields}/k} \ed_k(t).$$\end{definition} 

\begin{definition} For $G$ be a finite group, let $H^1(-;G):\text{Fields}/k \to \text{Sets}$ be defined by $H^1(-;G)(F/k) = \{\text{the isomorphism classes of } G\text{-torsors over }\text{Spec}F \}$. 
\end{definition}

\begin{definition} The \emph{essential dimension of} $\mathit{G}$ is
$$\ed_k(G) = \ed_k(H^1(-;G)).$$ \end{definition}

\begin{definition} Let $T: \text{Fields}/k \to \text{Sets}$ be a functor. Let $F/k$ be a field extension, and $t \in T(F)$. The \emph{essential} $\mathit{p}$\emph{-dimension of} $\mathit{t}$ is 
$$\ed_k(t,p) = \min \trdeg_k(F'')$$
where the minimum is taken over all
\begin{align*}
F'' \subset F' \text{ a finite extension}, \text{ with } F \subset F'\\
[F':F] \text{ finite } \text{ s.t. } p \nmid [F':F] \text{ and }\\
\text{the image of } t \text{ in } T(F') \text{ is in } \text{Im}(T(F'') \to T(F'))
\end{align*}
\end{definition}

\begin{note} ${\displaystyle \ed_k(t,p) = \min_{F \subset F', \text{ } p \text{ } \nmid \text{ }[F':F]} \ed_k(t|_{F'}).}$
\end{note}

\begin{definition} Let $T: \text{Fields}/k \to \text{Sets}$ be a functor. The \emph{essential} $\mathit{p}$\emph{-dimension of} $\mathit{T}$ is
$$\ed_{k}(T,p) = \sup_{t \in T(F), \text{ } F/k \in \text{Fields}/k} \ed_k(t,p).$$\end{definition}

\begin{definition} The \emph{essential} $\mathit{p}$\emph{-dimension of} $\mathit{G}$ is 
$$\ed_k(G,p) = \ed_k(H^1(-;G),p).$$ \end{definition}

\noindent The next lemma follows directly from the definitions:

\begin{lemma}\label{subseted}
If $H \subset G$, then $\ed_k(H,p) \leq \ed_k(G,p)$.
\end{lemma}

The key to proving the above lemma is that given a Galois $H$-algebra $E$ over $F$, we can extend to a Galois $G$-algebra over F. See the appendix (\ref{Ap2}) for the proof.

\begin{lemma}\label{Lempsyl} Let $S \in \syl_p(G)$. Then $\ed_k(G,p) = \ed_k(S,p).$  \end{lemma}

The key to proving the above lemma is that given a Galois $G$-algebra $E$ over $F$ there exists an extension of $F$, $F_0 = E^H$, such that $E$ is a Galois $H$-algebra over $E^H$.  See the appendix (\ref{Ap3}) for the proof.

The following lemma allows us to extend the underlying field $k$ when calculating essential $p$-dimension, so long as the extension is of degree prime to $p$. In particular, this allows us to assume our field $k$ contains $p$-th roots of unity (Corollary \ref{rootofunity}).

\begin{lemma}[\cite{KM}, Remark 4.8]\label{primetopext} If $k$ a field of characteristic $\neq p$, $k_1/k$ a finite field extension of degree prime to $p$, then 
$ \ed_{k}(G,p) = \ed_{k_1}(G,p).$
\end{lemma}

(The idea for the lemma above was brought to my attention by Federico Scavia and Zinovy Reichstein.) The key to proving Lemma \ref{primetopext} is the fact that given a field extension $F/k$ and a finite field extension $k_1/k$, $\trdeg_k(Fk_1) = \trdeg_k(F)$. See the appendix (\ref{Ap4}) for the proof. Putting Lemma \ref{primetopext} together with Lemma \ref{Lempsyl}, we get
\begin{corollary}\label{sylp} If $k_1/k$ a finite field extension of degree prime to $p$, $S \in \syl_p(G)$, then $\ed_k(G,p) = \ed_k(S,p) = \ed_{k_1}(S,p).$ \end{corollary}

\begin{corollary}\label{rootofunity} If $k$ a field of characteristic $\neq p$, $S \in \syl_p(G)$, $\zeta$ a primitive $p$-th root of unity, then
$$\ed_k(G,p) = \ed_{k(\zeta)}(S,p).$$ \end{corollary}

\begin{proof} Since $\zeta$ is a primitive $p$-th root of unity, $\zeta$ is a root of the polynomial $x^{p} - 1 = (x-1)(1 + \ldots + x^{p-1})$. Then the minimal polynomial over a field of characteristic prime to $p$ divides $1 + \ldots + x^{p-1}$ and so has degree prime to $p$.  So we have that $p \nmid [k(\zeta):k]$.  \end{proof}

\begin{note} By the corollary above, when calculating the essential $p$-dimension over a field $k$ of characteristic $\neq p$, we may assume that $k$ contains a primitive $p$-th root of unity. \end{note}

The following theorem and corollary from \cite {KM} will also be useful for our approach:
\begin{theorem}[Karpenko-Merkurjev \cite{KM}, Theorem 5.1]
\label{KM5.1} Let $G_1$ and $G_2$ be two $p$-groups, $k$ a field with $\text{char } k \neq p$ containing a primitive $p$-th root of unity, then $\ed_k(G_1\times G_2) = \ed_k(G_1) + \ed_k(G_2).$
\end{theorem}

\begin{corollary}\label{abpgp} Let $G$ be a finite abelian $p$-group, $k$ a field with $\text{char } k \neq p$ containing a primitive $p$-th root of unity. Then $\ed_k(G) = \text{rank}(G)$.\end{corollary}

\section{Representation Theory Background}

\begin{definition}\label{Def3.1} Let $H$ be a $p$-group. Define $\Omega_1(Z(H))$ (also called the socle of $H$) to be the largest elementary abelian $p$-group contained in $Z(H)$, i.e. $\Omega_1(Z(H)) = \{z \in Z(H) : z^p = 1\}$. \end{definition}

\begin{definition} For $G$ an abelian group, $k$ a field, let $\widehat{G}$ denote the group of characters of $G$ (homomorphisms from $G$ to $k^\times$). We will use the notation $\widehat{\Omega_1}(Z(H))$ for the character group of $\Omega_1(Z(H))$.\end{definition}

The next lemma is due to Meyer-Reichstein \cite{MR} and reproduced in \cite{BMS}.

\begin{lemma}[\cite{MR}, Lemma 2.3]\label{BMKS3.4} Let $k$ be a field with $\text{char } k \neq p$ containing $p$-th roots of unity. Let $H$ be a finite $p$-group and let $(\rho_i: H \to GL(V_i))_{1 \leq i \leq n}$ be a family of irreducible representations of $H$ with central characters $\chi_i$. Suppose that $\{\chi_i|_{\Omega_1(Z(H))} : 1 \leq i \leq n\}$ spans $\widehat{\Omega_1}(Z(H))$. Then $\bigoplus_i \rho_i$ is a faithful representation of $H$.
\end{lemma}

\begin{note} For each of the groups $S \in \syl_p(G)$ in this paper, $\Omega_1(Z(S)) = Z(S)$, so we can ignore the $\Omega_1$ for the purposes of this paper.
\end{note}

Let $\F_{p^r}^+ \cong (\Z/p\Z)^r$ denote the additive group of $\F_{p^r}$.

\begin{definition} For $k$ containing a $p$-th root of unity, fix a nontrivial character $\psi$ of $\F_{p^r}^+ \to k$. For $b \in \F_{p^r}$, define $\psi_b(x) = \psi(bx)$. \end{definition}

\begin{remark}\label{psib} The map given by $b \mapsto \psi_b$ is an isomorphism between $\F_{p^r}^+$ and $\widehat{\F_{p^r}^+}$. \end{remark}

We will use boldface $\mbf{b}$ to denote elements in $(\F_{p^r})^m$ and $b_1, b_2, \dots, b_m \in \F_{p^r}$ to denote the components.

\begin{definition} Fix a nontrivial character $\psi$ of $\F_{p^r}^+ \to k$. Fix $m$. For $\mbf{b} = (b_j) \in (\F_{p^r}^+)^{m}$, define 
$$\psi_{\mbf{b}}(\mbf{d}) = \prod_j (\psi_{b_j}(d_j)) \in \widehat{(\F_{p^r}^+)^{m}} ,$$ where $b_j, d_j$ are the components of $\mbf{b},\mbf{d}$. \end{definition}

\begin{lemma}\label{charform} For $k$ containing a $p$-th root of unity, fix a nontrivial character $\psi$ of $\F_{p^r}^+ \to k$. Then $\mbf{b} \mapsto \psi_{\mbf{b}}$ gives an isomorphism $(\F_{p^r}^+)^m \cong \widehat{(\F_{p^r}^+)^m}$, and $\psi_\mbf{b}(\mbf{d}) = \psi(\mbf{b}\mbf{d}^T)$. \end{lemma}

\subsection{The Wigner-Mackey Little Group Method}\label{WigMac}

The following exposition of Wigner-Mackey Theory follows \cite{Se} Section 8.2 and is also reproduced in \cite{BMS} page 7: Let $G$ be a finite group such that we can write $G = \Delta \rtimes L$ with $\Delta$ abelian. Let $k$ be a field with $\text{char } k \nmid |G|$ such that all irreducible representations of $\Delta$ over $k$ have degree 1. Then the irreducible characters of $\Delta$ form a group $\widehat{\Delta} = \text{Hom}(\Delta,k^\times)$. The group $G$ acts on $\widehat{\Delta}$ by 
$$(\chi^g)(a) = \chi(gag^{-1}), \text{ for } g \in G, \chi \in \widehat{\Delta}, a \in \Delta.$$
Let $(\psi_s)_{\psi_s \in \widehat{\Delta}/L}$ be a system of representatives for the orbits of $L$ in $\widehat{\Delta}$. For each $\psi_s$, let $L_s$ be the subgroup of $L$ consisting of those elements such that $l\psi_s = \psi_s$, that is $L_s = \text{Stab}_L(\psi_s)$. Let $G_s = \Delta \cdot L_s$ be the corresponding subgroup of $G$. Extend $\psi_s$ to $G_s$ by setting
$$\psi_s(al) = \psi_s(a), \text{ for } a \in \Delta, l \in L_s.$$
Then since $l\psi_s = \psi_s$ for all $l \in L_s$, we see that $\psi_s$ is a one-dimensional representation of $G_s$. Now let $\lambda$ be an irreducible representation of $L_s$; by composing $\lambda$ with the canonical projection $G_s \to L_s$ we obtain an irreducible representation $\lambda$ of $G_s$, i.e 
$$\lambda(al) = \lambda(l), \text{ for } a \in \Delta, l \in L_s.$$  
Finally, by taking the tensor product of $\chi_s$ and $\lambda$, we obtain an irreducible representation $\psi_s \otimes \lambda$ of $G_s$.  Let $\theta_{s,\lambda}$ be the corresponding induced representation of $G$, i.e. $\theta_{s,\lambda} := \text{Ind}_{G_s}^G (\psi_s \otimes \lambda).$ The following is an extension of Proposition 25 in Chapter 8 of \cite{Se}, it is called ``Wigner-Mackey theory'' in \cite{BMS} (Theorem 4.2):

\begin{theorem}[Venkataraman \cite{GV}, Theorem 4.1; Serre (for $k = \C$) \cite{Se}, Proposition 25]\label{WM}  Under the above assumptions, 
\begin{enumerate}[(i)]
\item $\theta_{s,\lambda}$ is irreducible.
\item Every irreducible representation of $G$ is isomorphic to one of the $\theta_{s,\lambda}$.
\end{enumerate}\end{theorem}

Venkataraman also proves a uniqueness statement: If $\theta_{s,\lambda}$ and $\theta_{s',\lambda'}$ are isomorphic, then $\psi_s = \psi_s'$ and $\lambda$ is isomorphic to $\lambda'$. But we do not care about the uniqueness of the irreducible representations.  In what follows, we will consider characters $\psi_s$ with $\psi_s \in \widehat{\Delta}$ rather than $\psi_s \in \widehat{\Delta}/L$. The two points above still hold.

 Note that in the cases considered in this paper, the conditions hold so long as $\text{char } k \neq p$.  Since we are considering the Sylow $p$-subgroups, this takes care of the first condition that $\text{char }k \nmid |G|$. All of our Sylow $p$-subgroups have the form $\Delta \rtimes L$ with $\Delta \cong (\Z/p\Z)^N$ for some $N > 0$. By the note following Lemma \ref{primetopext}, we may assume that $k$ contains a primitive $p$-th root of unity.  Thus we can conclude that all irreducible representations of $\Delta$ over $k$ have degree $1$.

 The dimension is given by $\dim(\theta_{s,\lambda}) = \frac{|L|}{|L_s|}\dim(\lambda)$. If we pick $\lambda = 1$, then this will minimize the dimension of the representation and we will have $\dim(\theta_{s,1}) = \frac{|L|}{|L_s|}.$ So for our purposes, we will only consider when $\lambda = 1$. The dimension of the representation will be minimized when $|L_s|$ is maximized.

\section{\texorpdfstring{The Case of the Linear Groups}{The Case of the Linear Groups}}

In this section, we will prove that

\begin{theorem}[\cite{BMS} $p \neq 2$, K. $p = 2$] \label{edPSLnallp} For any prime $p$, $k$ a field such that $\text{char } k \neq p$,
$$\ed_k(PSL_n(\F_{p^r}),p) = \ed_k(GL_n(\F_{p^r}),p) = rp^{r(n-2)}.$$
\end{theorem}

\noindent In this case, we will actually identify a subgroup (the Heisenberg subgroup) of a Sylow $p$-subgroup, to which Wigner-Mackey theory can be applied. This will give a lower bound for the essential $p$-dimension.  We will find an upper bound by constructing a specific faithful representation (we will extend the minimal dimensional representation of the Heisenberg subgroup to a representation of the same dimension).

\subsection{\texorpdfstring{The Sylow $p$-subgroups and their centers}{The Sylow p-subgroups and their centers}}

\begin{definition} Define $\Up_{n}(\F_{p^r})$ to be the unitriangular $n \times n$ matrices over $\F_{p^r}$ under multiplciation. (Unitriangular matrices are upper triangular matrices with $1$'s on the diagonal).\end{definition}

\noindent The kernel of the natural homomorphism $GL_n(\F_{p^r}) \to PSL_n(\F_{p^r})$ has order prime to p, so it maps the Sylow $p$-subgroups of $GL_n(\F_{p^r})$ isomorphically onto Sylow $p$-subgroups of $PSL_n(\F_{p^r})$, so it suffices to consider the Sylow $p$-subgroups of $GL_n(\F_{p^r})$. It is straightforward to show the following two lemmas.

\begin{lemma}\label{sylpn} For all $n \geq 2$ and all primes $p$, we have  $\Up_n(\F_{p^r}) \in\syl_p(GL_n(\F_{p^r}))$.\end{lemma}

\begin{lemma}\label{Zsylpn} For all $n \geq 2$ and all primes $p$, we have $$Z(\Up_n(\F_{p^r})) =  \{\begin{pmatrix} 1 & 0 & \ldots & 0 & b \\
0 & 1 & 0 & \ldots & 0\\
& & \ddots & & \vdots\\
0 & 0 & \ldots & 1 & 0\\
0 & 0 & 0 & \ldots & 1\end{pmatrix}\} \cong \F_{p^r}^+ \cong (\Z/p\Z)^r$$ \end{lemma}

\begin{lemma} $\Up_n(\F_{p^r}) \cong (\F_{p^r}^+)^ {n-1} \rtimes \Up_{n-1}$, where the action of $\Up_{n-1}$ on $(\F_{p^r}^+)^{n-1}$ is given by $A(\mbf{b}) = (A\mbf{b}^T)^T$. \end{lemma}

\begin{proof}
Let $$N = \{\begin{pmatrix} 1 & 0 & \ldots & 0 & b_{1} \\
0 & 1 & 0 & \ldots & b_2\\
& & \ddots & & \vdots\\
0 & 0 & \ldots & 1 & b_{n-1}\\
0 & 0 & 0 & \ldots & 1\end{pmatrix}\}$$ and let $$H =\{\begin{pmatrix} A & \mbf{0}\\
\mbf{0} & 1\end{pmatrix}, \text{ with } A \in \Up_{n-1}(\F_{p^r})\}.$$

Then $\Up_n(\F_{p^r}) \cong N \rtimes H$, where the action of $L$ on $N$ is given by 
$\begin{pmatrix} A & \mbf{0}\\
\mbf{0} & 1\end{pmatrix}\left(\begin{pmatrix} \text{Id} & \mbf{b}^T\\
\mbf{0} & 1\end{pmatrix} \right) = \begin{pmatrix} \text{Id} & A\mbf{b}^T\\
\mbf{0} & 1\end{pmatrix}$.
\end{proof}

Note that the center is given by $\{(\mbf{b},\text{Id})\}$, where $\mbf{b} = (b, 0, \cdots, 0)$.

\subsection{Classifying the irreducible representations}

By Corollary \ref{rootofunity}, we may assume that our field $k$ contains $p$-th roots of unity. We will use Wigner-Mackey Theory with $\Up_n(\F_{p^r}) \cong (\F_{p^r}^+)^{n-1} \rtimes \Up_{n-1}(\F_{p^r})$ to see what is the minimum dimension of an irreducible representation with non-trivial central character.  So we have 
$$\Delta = (\F_{p^r}^+)^{n-1}, \qquad L = \Up_{n-1}(\F_{p^r}).$$


Fix $\psi$ a non-trivial character of $\F_{p^r}^+$. By Lemma \ref{charform}, there is an isomorphism between $(\F_{p^r}^+)^{n-1}$ and $\widehat{(\F_{p^r}^+)^{n-1}}$ given by sending $\mbf{b} \in (\F_{p^r}^+)^{n-1}$ to the character $\psi_\mbf{b}$ defined by $\psi_{\mbf{b}}(\mbf{d}) = \psi(\mbf{b}\mbf{d}^T)$. A straightforward computation shows that for any prime $p$, the characters extending a non-trivial central character are $\psi_{\mbf{b}}$ with $b_1 \neq 0$. A straightforward calculation shows that $H \in L_\mbf{b}$ if and only if $\psi((\mbf{b}H - \mbf{b})\mbf{d}^T) = 1$ for all $\mbf{d} \in (\F_{p^r}^+)^{n-1}.$ (See the appendix \ref{4.2calc} for the details of this calculation.) Thus $H \in L_\mbf{b}$ if and only if $\mbf{b}H = \mbf{b}$.

Recall that $\theta_{\mbf{b},\lambda}$ refers to $\text{Ind}_{G_s}^G (\psi_s \otimes \lambda)$ and that in order to minimize the dimension, we are assuming $\lambda = 1$, in which case $\dim(\theta_{\mbf{b},1}) = \frac{|L|}{|L_{\mbf{b}}|}$, where $L_{\mbf{b}}$ is the stabilizer of $\psi_\mbf{b}$ in $L$.

\begin{proposition}\label{LsPSL} For all $p$, $\mbf{b} \in (\F_{p^r}^+)^{n-1}$ with $b_1 \neq 0$,

$$\dim(\theta_{\mbf{b},1}) = p^{r(n-2)}.$$
\end{proposition}

\begin{proof}

Let $b_1 \neq 0$. Note that $\mbf{b}H = 0$ if and only $\frac{1}{b_1}\mbf{b}H = 0$; thus we may assume without loss of generality that $b_1 = 1$. If $\mbf{b} = (1, 0, ,..., 0)$, then we have 
$$L_\mbf{b} = \{H \in \Up_{n-1}(\F_p) : H_{1,j} = 0 \text{ } \forall j \neq 1\} \cong \Up_{n-2}(\F_p).$$

Furthermore, for any $\mbf{b}'$ with $b_1' = 1 $, there exists $C \in \Up_{n-1}(\F_{p^r})$ such that  $\mbf{b}' = \mbf{b}C$ (namely the matrix $C$ such that $C_{1,j} = b_{j}$). Then $L_\mbf{b}$ and $L_{\mbf{b}'}$ are conjugate by $C$ and hence are isomorphic. Thus $L_{\mbf{b'}} \cong L_\mbf{b} \cong \Up_{n-2}(\F_{p^r})$.


And hence 
$$\dim(\theta_{\mbf{b},1}) = \frac{|L|}{|L_{\mbf{b}}|} = \frac{|\Up_{n-1}(\F_{p^r})|}{|\Up_{n-2}(\F_{p^r})|} = p^{r(n-2)}.$$

\end{proof}

\subsection{Proof of Theorem \ref{edPSLnallp}}

\begin{proof}
\text{}  By Corollary \ref{rootofunity}, we may assume that our field $k$ contains $p$-th roots of unity. So by Lemma \ref{BMKS3.5}, faithful representations of $S(p,n)$ of minimal dimension will decompose as a direct sum of exactly $r = \text{rank}(Z(S(p,n)))$ irreducible representations with non-trivial central characters. Since the center of $\Up_n(\F_{p^r})$ has rank $r$ and the dimension of any irreducible representation with non-trivial central character is $p^{r(n-2)}$,
$$\ed_k(\Up_n(\F_{p^r}),p) = rp^{r(n-2)}.$$

\end{proof}

\section{The Case of the Symplectic Groups}
In this section, we will show that
\begin{theorem}\label{edPSp}  For $k$ a field such that $\text{char } k \neq p$,
$$\ed_k(PSp(2n,p^r),p) = \ed_k(Sp(2n,p^r),p) = \begin{cases} rp^{r(n-1)}, &p \neq 2 \text{ or } n =2\\
r2^{r(n-1)-1}(2^{r(n-2)} + 1), &p = 2, n > 2 \end{cases}$$
\end{theorem}

 We do not prove the case $p = 2, n = 2, r = 1$, since it is already known that $\ed_k(PSp(4,2)',2) = \ed_k(A_6,2) = 2$. In any other case, $PSp(2n,p^r)' = PSp(2n,p^r)$, so we obtain a complete calculation of $\ed_k(PSp(2n,p^r)',p)$.

\subsection{Definitions}

\begin{definition} Let $S = \begin{pmatrix} 0 & \text{Id}_n\\
-\text{Id}_n & 0\end{pmatrix}$. The symplectic groups are defined by 
$$Sp(2n,p^r) := \{M \in GL_{2n}(\F_{p^r}) : M^T S M = S\},$$ and the projective symplectic groups are defined by 
$$PSp(2n,p^r) := Sp(2n,p^r)/Z(Sp(2n,p^r)).$$ \end{definition}

\noindent Note: A matrix $M = \begin{pmatrix} A & B\\
C & D \end{pmatrix} \in GL_{2n}(\F_{p^r})$ is symplectic if and only if $A^TC$, $B^TD$ are symmetric and  $A^TD - C^TB = \text{Id}_n$. 

\subsection{\texorpdfstring{The Sylow $p$-subgroups and their centers}{The Sylow p-subgroups and their centers}}

The kernel of the natural homomorphism  $Sp(2n,p^r) \to PSp(2n,p^r)$ has order prime to $p$, so it maps the Sylow $p$-subgroups of $Sp(2n,p^r)$ isomorphically onto Sylow $p$-subgroups of $PSp(2n,p^r)$, so it suffices to consider the Sylow $p$-subgroups of $Sp(2n,p^r)$.

\begin{definition} For any prime $p$, define $\sym(n,p^r)$ as the group of $n \times n$ symmetric matrices under addition (with entries from $\F_{p^r}$). \end{definition}

It is straightforward to show the following results. 
\begin{lemma}\label{sylPSp}[See \cite{Pre}, Lemma 1] For any prime $p$, let $$S(p,n) = \{\begin{pmatrix} A & 0_n\\
0_n & (A^{-1})^T\end{pmatrix}\begin{pmatrix} \text{Id}_n & B\\
0_n & \text{Id}_n\end{pmatrix} : A \in \Up_{n}(\F_{p^r}), B \in \sym(n,p^r)\}.$$ Then $S(p,n) \in \syl_p(Sp(2n,p^r))$. \end{lemma}

\begin{corollary}\label{Cor5.5}[See \cite{Pre95}] For any prime $p$, $S(p,n)$ the Sylow $p$-subgroup of $Sp(2n,p^r)$ defined in Lemma \ref{sylPSp}, 
$$S(p,n) \cong \sym(n,p^r) \rtimes \Up_{n}(\F_{p^r}),$$
where the action is given by $A(B) = ABA^T,$ where $B \in \sym(n,p^r), A \in \Up_n(\F_{p^r})$.
\end{corollary}

\begin{lemma}\label{Lem5.6}For $p \neq 2$, $S(p,n)$ the Sylow $p$-subgroup of $Sp(2n,p^r)$ defined in Lemma \ref{sylPSp}, $$Z(S(p,n)) = \{\begin{pmatrix} \text{Id}_n & D\\
0_n & \text{Id}_n\end{pmatrix} : D  = \begin{pmatrix} d & \mbf{0}\\
\mbf{0} & 0_{n-1} \end{pmatrix}\} \cong \F_{p^r}^+ \cong (\Z/p\Z)^r$$
\end{lemma}

\begin{lemma}\label{Lem5.7} For $S(2,n)$ the Sylow $p$-subgroup of $Sp(2n,2^r)$ defined in Lemma \ref{sylPSp}, \begin{align*}
Z(S(2,n)) &= \{\begin{pmatrix} \text{Id}_n & D\\
0_n & \text{Id}_n\end{pmatrix} : D_{i,j} = 0, \text{ for all } (i,j) \notin \{(1,1),(1,2),(2,1), D_{1,2} = D_{2,1}\} \cong (\F_{2^r}^+)^2 \cong (\Z/2\Z)^{2r}
\end{align*}
\end{lemma}
See the appendix (\ref{Ap5.67}) for the calculations of the centers.

\subsection{Classifying the irreducible representations}

By Corollary \ref{rootofunity}, we may assume that our field $k$ contains $p$-th roots of unity. We will use Wigner-Mackey Theory with $S(p,n) \cong \sym(n,p^r) \rtimes \Up_n(\F_{p^r})$ to compute the minimum dimension of an irreducible representation with non-trivial central character. So
$$\Delta = \sym(n,p^r), L =\Up_n(\F_{p^r}).$$ 
For $$B = \begin{pmatrix} b_1 & b_2 & \ldots & & b_{n}\\
b_2 & b_{n+1} & \ldots & & b_{2n-1}\\
\vdots & &\ddots & & \vdots\\
b_{n-1} & \ldots & & b_{n(n+1)/2-2} & b_{n(n+1)/2-1}\\
b_{n} & \ldots & & b_{n(n+1)/2-1} & b_{n(n+1)/2}\end{pmatrix} \in \sym(n,p^r),$$
let $\mbf{b} = (b_1,\ldots, b_{n(n+1)/2})$. Then the map map $B \mapsto \mbf{b}$ gives an isomorphism $\sym(n,p^r) \cong (\F_{p^r}^+)^{n(n+1)/2}$.

Fix $\psi$ a non-trivial character of $\F_{p^r}^+$. By Lemma \ref{charform}, there is an isomorphism between $(\F_{p^r}^+)^{n(n+1)/2}$ and $\widehat{(\F_{p^r}^+)^{n(n+1)/2}}$ given by sending $\mbf{b} \in (\F_{p^r}^+)^{n(n+1)/2}$ to the character $\psi_\mbf{b}$ defined by $\psi_{\mbf{b}}(\mbf{d}) = \psi(\mbf{b} \cdot \mbf{d})$. A straightforward computation shows that for $p \neq 2$, the characters extending a non-trivial central character are $\psi_{\mbf{b}}$ with $b_1 \neq 0$. Similarly, a straighforward computation shows that for $p = 2$, the characters extending a non-trivial central character are $\psi_{\mbf{b}}$ with $(b_1,b_2) \neq (0,0)$, that is $b_1 \neq 0$ or $b_2 \neq 0$.  Note that $H \in L_{\mbf{b}}$ if and only if $\psi(\mbf{b} \cdot (\mbf{hdh^T} - \mbf{d})) = 1$ for all $\mbf{d} \in (\F_{p^r})^{n(n+1)/2},$ where $\mbf{hdh^T}$ is the vector corresponding to $HDH^T$ under the isomorphism $\sym(n,p^r) \cong (\F_{p^r}^+)^{n(n+1)/2}$. See the appendix \ref{5.3calc} for the full details of the computation.

\subsubsection{\texorpdfstring{The case $p \neq 2$}{The case p neq 2}}

\begin{proposition}\label{Ls3} For $p \neq 2$,
$$\min_{\mbf{b} \in (\F_{p^r}^+)^{n(n+1)/2}, \text{ } b_1 \neq 0}\dim(\theta_{\mbf{b},1}) = p^{r(n-1)}.$$ This minimum is achieved when $\mbf{b} = (b,0,\ldots,0)$ with $b \neq 0$.\end{proposition}

\begin{proof}

Recall that $\mbf{b}, \mbf{d}$ are vectors corresponding to matrices $B, D \in \sym(n,p^r)$ via the isomorphism defined above for $\sym(n,p^r) \cong (\F_{p^r}^+)^{n(n+1)/2}$ and $\mbf{hdh^T}$ is the vector in $(\F_{p^r}^+)^{n(n+1)/2}$ corresponding to $HDH^T \in \sym(n,p^r)$ under the isomorphism $\sym(n,p^r) \cong (\F_{p^r}^+)^{n(n+1)/2}$.

We prove this proposition by showing that for $\mbf{b} = (b_1, \cdots, b_{n(n+1)/2})$ with $b_1 \neq 0$, $|L_{\mbf{b}}| \leq |\Up_{n-1}(\F_{p^r})| = p^{r(n-1)(n-2)/2}$. Pick $j_0 \neq 1$ and choose $D$ with $d_{i,j} = 0$ except for $d_{1,j_0}$ and let $\mbf{d}$ be the corresponding vector. Then
$$\mbf{b} \cdot (\mbf{hdh^T} - \mbf{d})  = d_{1,j_0}\left(2h_{1,j_0}B_{1,1} + \sum_{i=2}^{j_0-1} h_{i,j_0}B_{1,i}\right).$$ 
So since we need $\psi(\mbf{b} \cdot (\mbf{h}\mbf{d}\mbf{h}^T - \mbf{d})) = 1$ for all choices of $\mbf{d}$, we can conclude that
$$h_{1,j_0} = \frac{-1}{2B_{1,1}} \sum_{i=2}^{j_0-1} h_{i,j_0}B_{1,i}.$$
So
$$|L_{\mbf{b}}| \leq |\{H : H_{1,j} \text{ fixed } \forall j \neq 1\}| = |\Up_{n-1}(\F_{p^r})| = p^{r(n-1)(n-2)/2}$$
It is straightforward to show that for $\mbf{b} = (b,0,\ldots,0)$,
$$L_{\mbf{b}} = \{(0_n,H^{-1}) : H_{1,j} = 0, \forall j \neq 1\} \cong \Up_{n-2}(\F_{p^r}).$$
Thus the minimum is achieved when $\mbf{b} = (b,0,\ldots,0).$

\end{proof}

\subsubsection{\texorpdfstring{The case $p = 2$}{The case p = 2}}

\textbf{Case 1: } $\mbf{n = 2}$

\begin{proposition}\label{Ls5} For $p = 2$, $n = 2$, 
$$\min_{\mbf{b} \in (\F_{p^r}^+)^{3}, \text{ } b_1 \neq 0, b_2\neq 0}\dim(\theta_{\mbf{b},1})  = 2^{r-1}.$$ 
This minimum is achieved when $\mbf{b} = (b_1,b_2,0)$ with $b_1 \neq 0, b_2 \neq 0$.\\
If $\mbf{b} = (b_1, b_2, 0)$ with $b_1 \neq 0, b_2 \neq 0$, then 
$$\dim(\theta_{\mbf{b},1})  = 2^r.$$\end{proposition}

\begin{proof}
The proof is similar to that for $p \neq 2$.  We refer the reader to the appendix (\ref{Ap5.9}) for full details. \end{proof}

\noindent \textbf{Case 2: } $\mbf{n > 2}$

\begin{proposition}\label{Ls2} For $p = 2$, $n > 2$,
$$\min_{\mbf{b} \in (\F_{p^r}^+)^{n(n+1)/2}, \text{ } b_2 \neq 0}\dim(\theta_{\mbf{b},1}) =  2^{r(2n-3) -1}.$$
This minimum is achieved when $\mbf{b} = (b_i) = (b_1,b_2,0,\ldots, 0)$ with $b_1, b_2 \neq 0$. 
$$\min_{\mbf{b} \in (\F_{p^r}^+)^{n(n+1)/2}, \text{ } b_1 \neq 0}\dim(\theta_{\mbf{b},1})  =  2^{r(n-1) - 1}.$$
This minimum is achieved when $\mbf{b} = (b_i) = (b_1,0,b_3, \ldots, 0)$ with $b_1,b_3 \neq 0$.\end{proposition}

\begin{proof}
The proof is again similar. We refer the reader to the appendix (\ref{Ap5.10}) for this proof. \end{proof}

\noindent Note: For any $n > 2$ and any $r$, $2^{r(2n-3)-1} > 2^{r(n-1)-1}$.

\subsection{Proof of Theorem \ref{edPSp}}

\begin{proof}
\text{}  By Corollary \ref{rootofunity}, we may assume that our field $k$ contains $p$-th roots of unity. So by Lemma \ref{BMKS3.5}, faithful representations of $S(p,n)$ of minimal dimension will decompose as a direct sum of exactly $r = \text{rank}(Z(S(p,n)))$ irreducible representations.

\noindent \textbf{Case 1: } $\mbf{p \neq 2}$

\noindent Since the center of $S(p,n)$ has rank $r$ and the minimum dimension of an irreducible representation with non-trivial central character is $p^{r(n-1)}$,
$$\ed_k(PSp(2n,p^r),p) \geq rp^{r(n-1)}.$$

\noindent Let $\{e_i\}$ be a basis for $\F_{p^r}^+$ over $\F_p$, and let $s_i = (e_i,0,\ldots,0)$.  Let $\rho = \bigoplus_i \theta_{s_i,1}$ . Then by Proposition \ref{Ls3},
$$\dim(\rho) = rp^{r(n-1)}.$$
By Lemma \ref{BMKS3.4}, $\rho$ is a faithful representation of $S(p,n)$. Thus $$\ed_k(PSp(2n,p^r),p) = rp^{r(n-1)}.$$

\noindent \textbf{Case 2: } $\mbf{p = 2}$

\textbf{Step 1: Find the lower bound}


\qquad \textbf{Subcase 1:} $\mbf{n = 2}$: Since the center has rank $2r$ and by Proposition \ref{Ls5} the minimum dimension of an irreducible representation with non-trivial central character is $2^{r-1}$, 
$$\ed_k(PSp(4,2^r)) \geq 2r2^{r-1}  = r2^r.$$

\qquad \textbf{Subcase 2:} $\mbf{n > 2}$:  Let $\rho = \rho_i$ be a minimal dimensional faithful representation.  Since the set of all central characters $\{\chi_i\}$ must form a basis for $\widehat{(\F_{p^r}^+)^2}$, we can conclude that $b_2 \neq 0$ for at least $r$ of the $\rho_i$. So for these $\rho_i$ minimum dimension is $2^{r(2n-3)-1}$, by Proposition \ref{Ls2}. The other $r$ may have $b_2 = 0$, so their minimum dimension is $2^{r(n-1)-1}$, by Proposition \ref{Ls2}. Thus we have 
$$\ed_k\left(PSp(2n,2^r),2\right) \geq r2^{r(2n-3)-1} + r2^{r(n-1)-1} = r2^{r(n-1)-1}(2^{r(n-2)} + 1).$$

\textbf{Step 2: Construct the upper bound}

\noindent Let $\{e_i\}_{i=1}^{2r}$ be a basis for $\F_{2^r}^+$ over $\F_2$. Let $x$ be a nonzero element in $\F_{2^r}$. We will choose subsets $S$ of $\Delta = \sym(n,p^r)$ such that the set of all central characters of $\{\theta_{\mbf{b},1}\}_{\mbf{b} \in S}$ form a basis for the characters of the center. For $n = 2$, let $S = \{(e_i,e_i,0),(x,e_i,0)\}_{i=1}^{2r}.$ For $n > 2$, let 
$$S = \{(e_i,e_i,0,\ldots,0), (e_i,0,x,0, \ldots,0)\}_{i=1}^{2r}.$$
Let $\rho = \bigoplus_{\mbf{b} \in S} \theta_{\mbf{b},1}$. Then by Propositions \ref{Ls5} and \ref{Ls2},
$$\dim(\rho) = \sum_{\mbf{b} \in S} \dim(\theta_{\mbf{b},1}) = \begin{cases} r2^r, &n = 2, r > 1\\
r2^{r(n-1)-1}(2^{r(n-2)} + 1), &n > 2\end{cases}.$$
By Lemma \ref{BMKS3.4}, 
$\rho$ is a faithful representation of $S(2,n)$. 
Steps 1 and 2 together give us that $$\ed_k\left(PSp(2n,2^r),2\right) = \begin{cases} r2^r, &n = 2, r > 1\\
r2^{r(n-1)-1}(2^{r(n-2)} + 1), &n > 2\end{cases}.$$
\end{proof}

\section{The Case of the Orthogonal Groups}
In this section, we will show the following theorem:
\begin{theorem}\label{edO2n} For $\epsilon \in \{\pm\}$ in the notation of Subsection \ref{6.1}, $k$ a field such that $\text{char } k \neq p$,
\begin{align*} 
 \ed_k(P\Omega^\epsilon(2m,p^r),p) = \ed_k(\Omega^\epsilon(2m,p^r),p) = \begin{cases}  2r, &2m = 4, \text{ any } p\\
 rp^{2r(m-2)}, &2m > 4, \text{ any } p\\
\end{cases}\end{align*}
Furthermore, $\ed_k(O^\epsilon(2m,2^r),2) = 1 + \ed_k(\Omega^\epsilon(2m,2^r),2),$ and  for $p \neq 2$,  $\ed_k(O^\epsilon(2m,p^r),p) = \ed_k(\Omega^\epsilon(2m,p^r),p)$.
 \end{theorem}
 
 
\subsection{Definitions}\label{6.1}

\subsubsection{\texorpdfstring{The case $p \neq 2$}{The case p neq 2}}

Let $$A^+ = \begin{pmatrix} 0_m & \text{Id}_m\\
\text{Id}_m & 0_m\end{pmatrix}.$$
Let $\eta \in \F_{p^r}^\times$ be a non-square and let $$A^- = \begin{pmatrix} 0_{m-1} & \mbf{0} & \mbf{0} & \text{Id}_{m-1}\\
\mbf{0} & 1 & 0 & \mbf{0}\\
\mbf{0} & 0 & -\eta & \mbf{0}\\
\text{Id}_{m-1} & \mbf{0} & \mbf{0} &  0_{m-1}\end{pmatrix}.$$

\begin{definition} For $\epsilon \in \{\pm\}$, the orthogonal groups associated with $A^\epsilon$ are defined by 
$$O^\epsilon(2m,p^r) := \{M \in GL(2m,\F_{p^r}) : M^T A^\epsilon M = A^\epsilon\}.$$
The special orthogonal groups are defined by 
$$SO^\epsilon(2m,p^r) := \{M \in O^\epsilon(2m,p^r) : \det(M) = 1\}.$$ We define
$$\Omega^\epsilon(2m,p^r) := SO^\epsilon(2m,p^r)' \text{ (the commutator subgroup}).$$ Lastly, we define 
$$P\Omega^\epsilon(2m,p^r) := \Omega^\epsilon(2m,p^r)/(\Omega^\epsilon(2m,p^r) \cap \{\pm \text{Id}\}).$$ \end{definition}

\subsubsection{\texorpdfstring{The case $p = 2$}{The case p = 2}}

For $\mbf{x} = (x_i) \in \F_{p^r}^n$, let $Q^+(\mbf{x}) = \sum_{i=1}^m x_ix_{i+m}$, and let
 $$A_m^+ = \begin{pmatrix} 0_m & \text{Id}_m\\
0_m & 0_m\end{pmatrix}.$$
Then $Q^+(\mbf{x}) = \mbf{x}A_m^+\mbf{x}^T$. By Artin-Schreier theory, there exists $\eta \in \F_{2^r}$ such that $z^2 + z + \eta$ is irreducible in $\F_{2^r}[z]$. 

Let
$$Q_m^-(\mbf{x}) = \sum_{i=1}^m x_ix_{i+m} + x_m^2 + x_mx_{2m} + \eta x_{2m}^2$$ and define $A_m^-$ to be 

 $$A_m^-  = \begin{pmatrix} 0^1_m & \text{Id}_m\\
0_m & 0^\eta_m \end{pmatrix},\qquad \text{ where } 0^1_m = \begin{pmatrix} 0_{m-1} & \mbf{0}\\
\mbf{0} & 1\end{pmatrix} \text{ and } 0^\eta_m = \begin{pmatrix} 0_{m-1} & \mbf{0}\\
\mbf{0} & \eta\end{pmatrix}.$$
Then $Q_m^-(x) = \mbf{x}A_m^-\mbf{x}^T$. So if we write $\mbf{x} = (\mbf{a},b,\mbf{c},e)$ where $\mbf{a}, \mbf{c} \in \F_{2^r}^{m-1}, b,e \in \F_{2^r}$, then 
$$Q_m^-(\mbf{x}) = Q^+_{m-1}(\mbf{a},\mbf{c}) + b^2 + be + \eta e^2 = \mbf{a}\mbf{c}^T + b^2 + be + \eta e^2.$$ Or if we write $\mbf{x} = (\mbf{y},\mbf{z})$ where $\mbf{y},\mbf{z} \in \F_{2^r}^m$, then $$Q_m^-(\mbf{x}) = \mbf{y}\mbf{z}^T + y_m^2 + \eta z_m^2.$$

\begin{definition} Define $O^\epsilon(2m,2^r)$ as 
$$O^\epsilon(2m,2^r) := \{M \in GL(2m,\F_{2^r}) : Q^\epsilon(Mx) = Q^\epsilon(x) \text{ for all } x \in \F_{2^r}^{2m}\}.$$\end{definition}

\begin{definition} Define $B^\epsilon(x,y) = Q^\epsilon(x+y) + Q^\epsilon(x) + Q^\epsilon(y)$, the bilinear form corresponding to $Q^\epsilon$.\end{definition}

\noindent Note that $B^+(x,y) = \sum_{i=1}^m x_iy_{i+m} + \sum_{i=1}^m y_ix_{i+m}$.
So the corresponding matrix is 
$$S = \begin{pmatrix} 0 & \text{Id}_m\\
\text{Id}_m & 0\end{pmatrix}.$$
That is, $B^+(x,y) = xSy^T$. And $B^-(x,y) = \sum_{i=1}^{m-1} x_iy_{i+m} + y_ix_{i+m}  + x_{m}y_{2m} + y_{m}x_{2m}$, so the corresponding matrix is also $S$. That is, we have $B^-(x,y) = xSy^T = B^+(x,y)$, the same bilinear form as for $A^+$. Note that this is a nondegenerate alternating form and we have 

$$O^\epsilon(2m,2^r) \subset Sp(2m,2^r),$$ 
where $Sp(2m,2^r)$ is the symplectic group corresponding to $S$.

\begin{definition} Define $\Omega^\epsilon(2m,2^r) := O^\epsilon(2m,2^r)'$ (the commutator subgroup).\end{definition}

\noindent For consistency, we make the following definition:

\begin{definition} Define $P\Omega^\epsilon(2m,2^r) :=  \Omega^\epsilon(2m,2^r)/(\Omega^\epsilon(2m,2^r) \cap \{\pm \text{Id}\}) = \Omega^\epsilon(2m,2^r)$. \end{definition}

\begin{definition} The \emph{Dickson invariant}, $\delta^\epsilon_{2m,2^r}$, is a homomorphism from $O^\epsilon(2m,2^r)$ to $\Z/2\Z$ given by $\delta^\epsilon_{2m,2^r}(M) = \text{rank}(\text{Id}_{2m} - M) \mod 2$. Define 
$$SO^\epsilon(2m,2^r) := \ker \delta^\epsilon_{2m,2^r}.$$ \end{definition}

\begin{definition} Given $\epsilon \in \{\pm\}$, the Witt index $w_\epsilon$ is defined to be the dimension of a maximal totally isotropic subspace of $\F_{2^r}$ with respect to the quadratic form $Q^\epsilon$. \end{definition}

Grove shows (\cite{Gr}, Proposition 14.41) that for Witt index $w_\epsilon > 0$, and $n \geq 2$, 
$$\Omega^\epsilon(2m,2^r) = O^\epsilon(2m,2^r)' = SO^\epsilon(2m,2^r)'. $$ He also shows (\cite{Gr}, Theorem 14.43) that if $m \geq 2$ and $(m,w_\epsilon) \neq (2,2)$, then $\Omega^\epsilon(2m,2^r)$ is simple.

\subsection{\texorpdfstring{The Sylow $p$-subgroups and their centers}{The Sylow p-subgroups and their centers}}

\begin{definition} For any prime $p$, define $\asym(m,p^r)$ as the group of $m \times m$ anti-symmetric matrices under addition (with entries from $\F_{p^r}$).\end{definition}

\begin{definition}  For $p = 2$, define $\asymo(m,2^r) \subset \asym(m,2^r) = \sym(m,2^r)$ as the subgroup of symmetric/antisymmetric matrices with 0's on the diagonal. That is, $$\asymo(m,2^r) = \{B \in \sym(m,2^r) = \asym(m,2^r) : B_{i,i} = 0, \text{ } \forall i\}.$$\end{definition}

For $p \neq 2$, the Sylow $p$-subgroups of $P\Omega^+(2m,p^r)$, $\Omega^+(2m,p^r)$, and $O^+(2m,p^r)$ are isomorphic, so it suffices to consider the Sylow $p$-subgroups of $\Omega^+(2m,p^r)$. (We do this for notational purposes so we can combine the arguments with the case $p = 2$.) A direct computation shows the following.

\begin{lemma}\label{Osylp}[See \cite{Pre}, \cite{Mar}] For $p \neq 2$, let 
$$S^+(p,2m) = \{\begin{pmatrix} A & 0_m\\
0_m & (A^{-1})^T\end{pmatrix}\begin{pmatrix} \text{Id}_m & B\\
0_m & \text{Id}_m\end{pmatrix} : A \in \Up_{m}(\F_{p^r}), B \in \asym(m,p^r)\}.$$ 
Then $S^+(p,2m)$ is isomorphic to the elements in $\syl_p(\Omega^+(2m,p^r))$. 
\end{lemma}

\begin{corollary} For $p \neq 2$, $S^+(p,2m)$ as defined in Lemma \ref{Osylp}, 
$$S^+(p,2m) \cong \asym(m,p^r) \rtimes \Up_{m}(\F_{p^r}),$$
where the action is given by $A(B) = ABA^T.$
\end{corollary}
\begin{lemma}\label{O-syl} The Sylow $p$-subgroups of $O^+(2m,p^r)$ and $O^-(2m,p^r)$ are isomorphic.  \end{lemma}

\begin{proof}
For the proof, see the appendix (\ref{Ap5}).
\end{proof}

Since $S^+(p,2m) \cong S^-(p,2m)$, it suffices to consider $S^+(p,2m)$. 
For the sake of simplicity of notation, let $S(p,2m) = S^+(p,2m)$.

\begin{lemma}\label{Osyl2} Let  
$$S(2,2m) = \{\begin{pmatrix} A & 0_m\\
0_m & (A^{-1})^T\end{pmatrix}\begin{pmatrix} \text{Id}_m & B\\
0_m & \text{Id}_m\end{pmatrix} : A \in \Up_{m}(\F_{2^r}), B \in \asymo(m,2^r)\}.$$ Then $S(2,2m) \in \syl_2(\Omega^\epsilon(2m,2^r))$ for $\epsilon \in \{\pm\}$. \end{lemma}

\begin{corollary} For $S(2,2m)$ as defined in Lemma \ref{Osyl2},
$$S(2,2m) \cong \asymo(m,2^r) \rtimes \Up_m(\F_{2^r}),$$ where the action is given by 
$A(B) = ABA^T.$\end{corollary}

The above lemma is slightly more involved; see the appendix (\ref{Ap6}) for the details. For $O^\epsilon(2m,2^r)$, note that $\langle -\text{Id} \rangle \times S(2,2m)$ is a Sylow $2$-subgroup of $O^\epsilon(2m,2^r)$. Thus $\ed_k(O^\epsilon(2m,2^r),2) = 1 + \ed_k(\Omega^\epsilon(2,2^r),2).$

For $n = 4$,  the action of $\Up_2(\F_{p^r}) \cong \F_{p^r}$ on $\asym(2,p^r) \cong \F_{p^r}$ is trivial and so $S(p,n) \cong \F_{p^r} \times \F_{p^r}$. Thus the Sylow $p$-subgroup is abelian.

\begin{lemma}\label{Lem6.18} For any prime $p$, $m > 2$, let $S(p,2m) = S^+(p,2m)$ be defined as in Lemmas \ref{Osylp} and \ref{Osyl2}. Then $$Z(S(p,2m)) = \{\begin{pmatrix} \text{Id}_m & D\\
0_m & \text{Id}_m\end{pmatrix} : D = \begin{pmatrix} 0 & x & \mbf{0}\\
-x & 0 & \mbf{0}\\
\mbf{0} & \mbf{0} & 0_{m-2}\end{pmatrix}\} \cong \F_{p^r}^+ \cong (\Z/p\Z)^r$$
\end{lemma}

For the calculation of the centers, see the appendix (\ref{Ap7}).

\subsection{Classifying the irreducible representations}

By Corollary \ref{rootofunity}, we may assume that our field $k$ contains $p$-th roots of unity. We will use Wigner-Mackey Theory with 
$$S(p,2m) \cong \begin{cases} \asym(m,p^r) \rtimes \Up_m(\F_{p^r}) &p \neq 2\\
\asymo(m,2^r) \rtimes \Up_m(\F_{2^r}) &p = 2\end{cases}$$
 to see what is the minimum dimension of an irreducible representation with non-trivial central character. So
 $$\Delta = \begin{cases} \asym(m,p^r) &p \neq 2\\
\asymo(m,2^r) &p = 2\end{cases} \cong (\F_{p^r}^+)^{m(m-1)/2}, \qquad L = \Up_m(\F_{p^r}).$$
For $$B = \begin{pmatrix} 0 & b_1 & \cdots & & b_{m-1}\\
-b_1 &  0 & b_{m} & \cdots & b_{2m-3}\\
\vdots & &\ddots & & \vdots\\
-b_{m-2} & \cdots & & 0 & b_{m(m-1)/2}\\
-b_{m-1} & \cdots & & -b_{m(m-1)/2} & 0\end{pmatrix} \in \begin{cases} \asym(m,p^r), &p \neq 2\\
\asymo(m,p^r), &p = 2 \end{cases} $$
let $\mbf{b} = (b_1,\cdots, b_{m(m-1)/2})$. (When $p = 2$, the negatives go away.) Then the map $B \mapsto \mbf{b}$ gives an isomorphism $\begin{cases} \asym(m,p^r), &p \neq 2\\
\asymo(m,p^r), &p = 2\end{cases} \cong (\F_{p^r}^+)^{m(m-1)/2}.$

Fix $\psi$ a non-trivial character of $\F_{p^r}^+$. By Lemma \ref{charform}, there is an isomorphism between $(\F_{p^r}^+)^{m(m-1)/2}$ and $\widehat{(\F_{p^r}^+)^{m(m-1)/2}}$ given by sending $\mbf{b} \in (\F_{p^r}^+)^{m(m-1)/2}$ to the character $\psi_\mbf{b}$ defined by $\psi_{\mbf{b}}(\mbf{d}) = \psi(\mbf{b} \cdot \mbf{d})$. As for the symplectic groups, a straightforward computation shows that for any prime $p$, the characters extending a non-trivial central character are $\psi_{\mbf{b}}$ with $b_1 \neq 0$. Note that $H \in L_\mbf{b}$ if and only if $\psi(\mbf{b} \cdot (\mbf{hdh^T} - \mbf{d})) = 1$ for all $\mbf{d} \in (\F_{p^r}^+)^{m(m-1)/2},$
where $\mbf{hdh^T}$ is the vector in $(\F_{p^r}^+)^{m(m-1)/2}$ corresponding to $HDH^T \in \sym(m,p^r)$ under the isomorphsim $\sym(m,p^r) \cong (\F_{p^r}^+)^{m(m-1)/2}$. See the appendix(\ref{6.4calc}) for the full details of the computation.

\begin{proposition}\label{Ls1} For any prime $p$,
$$\min_{\mbf{b} \in  (\F_{p^r}^+)^{m(m-1)/2}, \text{ } b_1 \neq 0} \dim(\theta_{\mbf{b},1}) =  p^{2r(m-2)}.$$
This minimum is achieved when $\mbf{b} = (b,0,\ldots,0)$ with $b \neq 0$.\end{proposition}

\begin{proof}

Recall that $\mbf{b}, \mbf{d}$ are vectors corresponding to matrices $B, D \in \Delta$ via the isomorphism $\Delta \cong (\F_{p^r}^+)^{m(m-1)/2}$ and $\mbf{hdh^T}$ is the vector in $(\F_{p^r}^+)^{m(m-1)/2}$ corresponding to $HDH^T \in \asym(m,p^r)$ under the isomorphism $\asym(m,p^r) \cong (\F_{p^r}^+)^{m(m-1)/2}$.

\begin{calc'} For $j_0 > 2$, choosing $d_{i,j} = 0$ except for $d_{1,j_0} = -d_{j_0,1}$ and performing similar calculations to those for Propostion \ref{Ls3}, we get that
$$\sum_{i=2}^{j_0-1} h_{i,j_0}B_{1,i} = 0.$$ 
For $2 \leq k \leq n$, if $B_{1,k} \neq 0$, we can solve for $h_{k,j_0}$ in terms of $h_{i,j_0}$ for $i \neq 1,k$. If particular, since $B_{1,2} = b_1 \neq 0$, we can solve for $h_{2,j_0}$ in terms of $h_{i,j_0}$ with $i > 2$.
\end{calc'} 
\bigskip

\begin{calc'} For $j_0 > 2$, choose $d_{i,j} = 0$ except for $d_{2,j_0} = -d_{j_0,2}$, and again performing similar calculations to those for Propostion \ref{Ls3}, we get
$$-B_{1,2}h_{1,j_0} + \sum_{i=2}^{j_0} B_{1,i}h_{i,j_0}h_{1,2} + \sum_{i=3}^{j_0-1} B_{2,i}h_{i,j_0} = 0.$$
Since $B_{1,2} = b_1 \neq 0$, we can solve for $h_{1,j_0}$ in terms of $h_{1,2}$ and $h_{i,j_0}$ with $i > 2$.

\end{calc'}

\noindent Putting these two calculations together, we can conclude that for all $\mbf{b} = (b_i)$ with $b_1 \neq 0$,
$$|L_{\mbf{b}}| \leq |\{H : H_{2,j} \text{ fixed }, \forall j > 2, H_{1,j} \text{ fixed }, \forall j > 2\}| = |\F_{p^r}| \cdot |U_{m-2}(\F_{p^r})| = p^{r[(m-2)(m-3)/2 + 1]}.$$

We leave to the reader the verification that the minimum is achieved for $\mbf{b} = (b,0,\ldots,0)$.
\end{proof}

For more details of the above proof, see the appendix (\ref{Ap8}).

\subsection{Proof of Theorem \ref{edO2n}}

\begin{proof}
By Lemma \ref{BMKS3.5}, faithful representations of $S(p,2m)$ of minimal dimension will decompose as a direct sum of exactly $r = \text{rank}(Z(S(p,2m)))$ irreducible representations. We will complete the proof for four separate cases.

\noindent \textbf{Case 1:} $\mbf{2m = 4}$

For $p \neq 2$, the action of $\Up_2(\F_{p^r}) \cong \F_{p^r}$ on $\asym(2,p^r) \cong \F_{p^r}^+$ is trivial, and so $S^+(p,4) \cong \F_{p^r}^+ \times \F_{p^r}^+$.  So $\ed_k(S^+(p,4)) = \ed_k(\F_{p^r}^+ \times \F_{p^r}^+) = 2r$. 

Similarly for $n = 4$, $p = 2$, $S^+(2,4) \cong \F_{2^r} \times \F_{2^r}^+$.  So $\ed_k(S^+(2,4)) = \ed_k(\F_{2^r}^+ \times \F_{2^r}^+) = 2r$.

Note: The work in the previous section is valid, though unnecessary, for $n = 4$. It gives us that the minimum dimension of an irreducible representation is $1$. Then since the center has rank $2r$, we will get an essential dimension of $2r$.

\noindent \textbf{Case 2:} $\mbf{2m > 4}$

\noindent Since the center has rank $r$ and the minimum dimension of an irreducible representation with non-trivial central character is $p^{2r(m-2)}$, $$\ed_k(\Omega^+(2m,p^r),p) \geq rp^{2r(m-2)},$$

\noindent Let $\{e_i\}$ be a basis for $\F_{p^r}^+$ over $\F_p$, and let $s_i = (e_i,0,\ldots,0)$.  Let $\rho = \bigoplus_i \theta_{s_i,1}$ . Then by Proposition \ref{Ls1}, 
$$\dim(\rho) = \sum_{i=1}^r \dim(\theta_{s_i,1}) = rp^{2r(m-2)}.$$
By Lemma \ref{BMKS3.4}, $\rho$ is a faithful representation of $S^+(p,2m)$. Therefore
 $$ed_k(\Omega^\epsilon(2m,p^r),p) = rp^{2r(m-2)}.$$

%
%
%
\end{proof}

\newpage

\section{Appendix}

In this appendix, we provide some details for the computations in this article.

\subsection{Remark 4}\label{Ap1} 

Remark 4: Duncan and Reichstein calculated the essential $p$-dimension of the pseudo-reflection groups: For $G$ a pseudo-reflection group with $k[V]^G = k[f_1,\cdots,f_n]$, $d_i = \text{deg}(f_i),$ $\ed_k(G,p) = a(p) = |\{i : d_i \text{ is divisible by } p\}|$ (\cite{DR}, Theorem 1.1). These groups overlap with the groups above in a few small cases (The values of $d_i$ are in \cite{ST}, Table VII): 
\begin{enumerate}[(i)]
\item Group 12 in the Shephard-Todd classification, $Z_2.O \cong GL_2(\F_3)$: $d_1, d_2$ are $6,8$; so 
$$\ed_k(Z_2.O,3) = 1 = \ed_k(GL_2(\F_3),3).$$

\item Group 23 in the Shephard-Todd classification, $W(H_3) \cong \Z/2\Z \times PSL_2(\F_5)$: $d_1, \cdots d_3$ are $2,6,10$; so 
$$\ed_k(W(H_3),5) = 1 = \ed_k(PSL_2(\F_5),5).$$

\item Group 32 in the Shephard-Todd classification, $W(L_4) \cong \Z/3\Z \times Sp(4,3)$: $d_1, \cdots d_4$ are $12,18,24,30$; so 
$$\ed_k(W(L_4),3) = 4 = 1 + \ed_k(Sp(4,3),3).$$

\item Group 33 in the Shephard-Todd classification, $W(K_5) \cong \Z/2\Z \times PSp(4,3) \cong \Z/2\Z \times PSU(4,2)$: $d_1, \cdots d_5$ are $4, 6, 10, 12, 18$; so 
$$\ed_k(W(K_5),3) = 3 = \ed_k(PSp(4,3),3)$$
and 
$$\ed_k(W(K_5),2) = 5 = 1 +\ed_k(PSU(4,2)).$$

\item Group 35 in the Shephard-Todd classification, $W(E_6) \cong O^-(6,2)$: 
$d_1, \cdots, d_6$ are $2,5,6,8,9,12$; so 
$$\ed_k(W(E_6),2) = 4 = \ed_k(O^-(6,2),2).$$

\item Group 36 in the Shephard-Todd classification, $W(E_7) \cong \Z/2\Z \times Sp(6,2)$: $d_1, \cdots, d_7$ are $2,6,8,10,12,14,18$; so 
$$\ed_k(W(E_7),2) = 7 = 1 + \ed_k(Sp(6,2),2).$$
\end{enumerate}

\subsection{Lemma 2.8}\label{Ap2}
\begin{lma}[\ref{subseted}]
If $H \subset G$, then $\ed_k(H,p) \leq \ed_k(G,p)$.
\end{lma}
\begin{proof}

\begin{align*}
\ed_k(G,p) &= \ed_k(H^1(-;G))\\
&= \sup_{E \text{ Galois } G\text{-algebra over }F, \text{ } F/k \in \text{Fields}/k} \ed_k(E)
\end{align*}
And
\begin{align*}
\ed_{k}(G,p) &= \ed_k(H^1(-;G),p)\\
&= \sup_{E \text{ Galois } G\text{-algebra over }F, \text{ } F/k \in \text{Fields}/k} \ed_k(E,p)\\
&= \sup_{E \text{ Galois } G\text{-algebra over }F} \left(\min \trdeg_k(F''))\right)
\end{align*}
where the minimum is taken over all 
\begin{align*}
F'' \subset F' \text{ a finite extension}, \text{ with } F \subset F'\\
[F':F] \text{ finite } \text{ s.t. } p \nmid [F':F] \text{ and }\\
EF' = E'F'' \text{ for some } E' \text{ Galois } G\text{-algebra over } F''
\end{align*}

Thus 

\begin{align*}
&\ed_k(G,p)\\
&= \sup_{E \text{ Galois } G\text{-algebra over }F}\\
&\qquad \min_{F \subset F' \text{ a finite extension } \text{ and } p \nmid [F':F]}\\
&\qquad \qquad \min_{F'' \text{ s.t. } EF' = E'F'' \text{ for some } E' \text{ Galois } G\text{-algebra over } F''}  \trdeg_k(F''))
\end{align*}
And similarly,
\begin{align*}
&\ed_k(H,p)\\
&= \sup_{E \text{ Galois } H\text{-algebra over }F}\\
&\qquad \min_{F \subset F' \text{ a finite extension } \text{ and } p \nmid [F':F]}\\
&\qquad \qquad \min_{F'' \text{ s.t. } EF' = E'F'' \text{ for some } E' \text{ Galois } H\text{-algebra over } F''}  \trdeg_k(F''))
\end{align*}

Since $H$ is a subgroup of $G$, we have that given a Galois $H$-algebra $E$ over $F$, we can extend to a Galois $G$-algebra over F.  Thus it suffices to show that for $E \subset E_1$  with $E$ a Galois $H$-algebra over $F$ and $E_1$ a Galois $G$-algebra over $F$, if $F \subset F'$ is a finite extension with $p \text{ } \nmid \text{ } [F':F]$, then
\begin{align*}
&\min_{F'' \text{ s.t. } EF' = E'F'' \text{ for some } E' \text{ Galois } H\text{-algebra over } F''}  \trdeg_k(F''))\\
 &\leq \min_{F'' \text{ s.t. } E_1F' = E_1'F'' \text{ for some } E_1' \text{ Galois } G\text{-algebra over } F''}  \trdeg_k(F''))
\end{align*} 
Let $F \subset F'$ be a finite extension with $p \text{ } \nmid \text{ } [F':F]$. If $F''$ is such that there exists $E_1'$ with $E_1F' = E_1'F''$, then there exists a Galois $G$ algebra $E'$ over $F''$ contained in $E_1'F'$ such that $E_0F'' = E'F'$. 
Let $E' = E_0 \cap E$.  Then $E'$ is a Galois $H$-algebra over $F''$. Hence $F''$ is considered in the min for $\ed_\C(H,p)$.
Thus the desired inequality holds. Therefore,
$$\ed_k(H,p) \leq \ed_k(G,p).$$

\end{proof}

\subsection{Lemma 2.9}\label{Ap3}

\begin{lma}[\ref{Lempsyl}] Let $S \in \syl_p(G)$. Then $\ed_k(G,p) = \ed_k(S,p).$  \end{lma}

\begin{proof}

By Lemma 2.8, we already have $\ed_k(S,p) \leq \ed_k(G,p)$. So we only need to show that $\ed_k(G,p) \leq \ed_k(S,p)$. Since $S$ is a subgroup of $G$, we have that given a Galois $G$-algebra $E$ over $F$ there exists an extension of $F$, $F_0 = E^S$, such that $E$ is a Galois $S$-algebra over $E^S$. 
Thus it suffices to show that for $E$ a Galois $G$-algebra over $F$, which is also a Galois $S$-algebra over $F_0 = E^S$,

\begin{align*}
&\ed_k(G,p)\\
&\min_{F \subset F' \text{ a finite extension } \text{ and } p \nmid [F':F]}\\
&\qquad \min_{F'' \text{ s.t. } EF' = E'F'' \text{ for some } E' \text{ Galois } G\text{-algebra over } F''}  \trdeg_k(F''))\\
\leq &\min_{F_0 \subset F' \text{ a finite extension } \text{ and } p \nmid [F':F_0]}\\
&\qquad \min_{F'' \text{ s.t. } EF' = E'F'' \text{ for some } E' \text{ Galois } S\text{-algebra over } F''}  \trdeg_k(F''))
\end{align*}

Note that since $S$ is a subgroup of $G$ of index prime to $p$ and $[F_0:F] = [E^S :F] = [G:S]$, we get that $p \text{ } \nmid \text{ } [F_0:F]$.  
Given $F_0 \subset F' \text{ a finite extension } \text{ and } p \nmid [F':F_0]$, then

$$p \text{ } \nmid \text [F':F] = [F':F_0][F_0:F].$$

Thus $F'$ is also considered in the minimum for $\ed_k(G,p)$, and so the desired inequality holds. Therefore,
$$\ed_k(G,p) \leq \ed_k(H,p).$$

\end{proof}

\subsection{Lemma 2.10}\label{Ap4}
\begin{lma}[\ref{primetopext}; \cite{KM}, Remark 4.8] If $k$ a field of characteristic $\neq p$, $k_1/k$ a finite field extension of degree prime to $p$, then 
$ \ed_{k}(G,p) = \ed_{k_1}(G,p).$
\end{lma}

\begin{proof}
$T:\text{Fields}/k \to \text{Sets}$ be defined by $T(F/k) = $ the isomorphism class of $G$-torsors over $\text{Spec}F$. Recall that

\begin{align*}\ed_k(G,p) &= \sup_{t \in T(F), F/k \in \text{Fields}/k} \ed_k(t,p)\\
&= \sup_{t \in T(F), F/k \in \text{Fields}/k} \left( \min_{F'' \subset F' \text{ s.t. } p \nmid [F':F''] \text{ and the image of } t \text{ in } T(F') \text{ is in } \text{Im}(T(F'') \to T(F'))} \trdeg_k(F'') \right)   
\end{align*}

\noindent First we will show that $\mbf{\ed_{k_1}(G,p) \leq \ed_{k}(G,p)}$:

Let $F_1/k_1$, $t_1 \in T(F)$. We want to show that there exist $F/k$, $t \in T(F)$ such that 
$$\ed_{k_1}(t_1,p) \leq \ed_{k}(t,p).$$ 

In other words, if we are given $F'' \subset F'$ such that $p \nmid [F':F'']$, the image of $t$ in $T(F')$ is in $\text{Im}(T(F'') \to T(F'))$, we need to be able to show that there exists $F_1'' \subset F_1'$ such that $p \nmid [F_1':F_1'']$ and the image of $t_1$ in $T(F_1')$ is in $\text{Im}(T(F_1'') \to T(F_1'))$ and $$\trdeg_{k_1}(F_1'')\leq  \trdeg_{k}(F'') .$$

So, let $F = F_1$ and $t = t_1$. Suppose we are given $F'' \subset F'$ such that $p \nmid [F':F'']$ and the image of $t$ in $T(F')$ is in $\text{Im}(T(F'') \to T(F'))$. In other words, there exists $t_2 \in T(F'')$, $t_3 \in T(F')$. such that $t_2 \text{ and } t_1$ both map to $t_3$. 
Then let $F_1'' = F''k_1, F_1' = F'k_1$.  Then since $p \nmid [k_1:k]$ and $G$ is a $p$-group, $t_2k_1 \in T(F_1''), t_3k_1 \in T(F_1'),$ and $t_1$ and $t_2k_1$ both map to $t_3k_1$ in $T(F_1')$.  Since $[F_1':F_1''] \divides [F':F'']$ and $p \nmid [F':F'']$, we have that $p \nmid [F_1':F_1'']$. Also the image of $t$ in $T(F_1')$ is in $\text{Im}(T(F_1'') \to T(F_1'))$. Moreover, $\trdeg_{k_1}F_1'' = \trdeg_k F''$.  

Therefore, we can conclude that $\mbf{\ed_{k_1}(T,p) \leq \ed_{k}(T,p)}$.

\bigskip

\noindent Now we will show that $\mbf{\ed_k(G,p) \leq \ed_{k_1}(G,p)}:$

Let $F/k$, $t \in T(F)$. We want to show that there exist $F_1/k_1$, $t_1 \in T(F')$ such that 
$$\ed_k(t,p) \leq \ed_{k_1}(t_1,p).$$ 

In other words, if we are given $F_1'' \subset F_1'$ such that $p \nmid [F_1':F_1'']$ and the image of $t_1$ in $T(F_1')$ is in $\text{Im}(T(F_1'') \to T(F_1'))$, we need to be able to show that there exists $F'' \subset F'$ such that $p \nmid [F':F'']$, the image of $t$ in $T(F')$ is in $\text{Im}(T(F'') \to T(F'))$ and $$\trdeg_{k}(F'') \leq \trdeg_{k_1}(F_1'').$$

So, let $F_1 = Fk_1$ and let $t_1$ be the image of $t$ in $T(F_1)$. Suppose we are given $F_1'' \subset F_1'$ such that $p \nmid [F_1':F_1'']$ and the image of $t_1$ in $T(F_1')$ is in $\text{Im}(T(F_1'') \to T(F_1'))$. Then let $F'' = F_1'', F' = F_1'$. Then $p \nmid [F' : F''] = [F_1':F_1'']$, and the image of $t$ in $T(F')$ is the image of $t_1$ in $T(F_1')$ (from $T(F_1)$), which is in $\text{Im}(T(F'') \to T(F'))$. Moreover $\trdeg_{k} F'' = \trdeg_{k} F_1'' = \trdeg_{k_1} F_1''$, since $k_1/k$ is a finite extension.  

Therefore, we can conclude that $\mbf{\ed_k(T,p) \leq \ed_{k_1}(T,p)}$.

\end{proof}

\subsection{Section 4.2 Calculation}\label{4.2calc}

We are using Wigner-Mackey Theory to analyze the irreducible representations of 
$$\Up_n(\F_{p^r}) \cong (\F_{p^r}^+)^{n-1} \rtimes \Up_{n-1}(\F_{p^r}).$$ So we have 
$$\Delta = (\F_{p^r}^+)^{n-1}, \qquad L = \Up_{n-1}(\F_{p^r}).$$
Let $$N = \{\begin{pmatrix} 1 & 0 & \ldots & 0 & b_{1} \\
0 & 1 & 0 & \ldots & b_2\\
& & \ddots & & \vdots\\
0 & 0 & \ldots & 1 & b_{n-1}\\
0 & 0 & 0 & \ldots & 1\end{pmatrix}\}, \qquad H =\{\begin{pmatrix} A & \mbf{0}\\
\mbf{0} & 1\end{pmatrix}, \text{ with } A \in \Up_{n-1}(\F_{p^r})\}.$$
The isomorphism $\Up_n(\F_{p^r}) \cong (\F_{p^r}^+)^{n-1} \rtimes \Up_{n-1}$ is given by $N \mapsto (\F_{p^r}^+)^{n-1}$ via

$$\{\begin{pmatrix} 1 & 0 & \ldots & 0 & b_{1} \\
0 & 1 & 0 & \ldots & b_2\\
& & \ddots & & \vdots\\
0 & 0 & \ldots & 1 & b_{n-1}\\
0 & 0 & 0 & \ldots & 1\end{pmatrix}\} \mapsto \mbf{b}$$
and $$\begin{pmatrix} A & \mbf{0}\\
\mbf{0} & 1\end{pmatrix} \mapsto A.$$

Recall that $L_\mbf{b}$ is defined to the be the stabilizer in $L$ of $\psi_\mbf{b}$ under conjugation. That is, $H \in L_{\mbf{b}}$ if and only if for all $\mbf{d} \in (\F_{p^r}^+)^{n-1}$,
$$\psi_{\mbf{b}}((0_n,H)(\mbf{d},\text{Id})(0_n,H^{-1})) = \psi_{\mbf{b}}(\mbf{d}, \text{Id}).$$
(Note here we are viewing $\psi_\mbf{b}$ as a map on $\{(\mbf{d}, \text{Id}) : \mbf{d} \in \Delta\} \subset \Delta \rtimes L\}$.) 

Note that $(0_n,H)(\mbf{d},\text{Id}_n)(0_n,H^{-1})$ corresponds to 
\begin{align*}
&\begin{pmatrix} H & \mbf{0}\\
0 & 1 \end{pmatrix} \begin{pmatrix} \text{Id} & \mbf{d}^T\\
\mbf{0} & 1\end{pmatrix} \begin{pmatrix} H^{-1} & \mbf{0}\\
\mbf{0} & 1 \end{pmatrix}\\
&= \begin{pmatrix} H & H\mbf{d}^T\\
\mbf{0} & 1\end{pmatrix} \begin{pmatrix} H^{-1} & \mbf{0}\\
\mbf{0} & 1 \end{pmatrix}\\
&= \begin{pmatrix} \text{Id} & H\mbf{d}^T\\
\mbf{0} & 1\end{pmatrix} 
\end{align*}
which corresponds to $(H\mbf{d}^T)^T$ in $\Delta = (\F_{p^r}^+)^{n-1}$. Then, viewing $\psi_\mbf{b}$ as a map on $\Delta$, we see that $H \in L_\mbf{b}$ if and only if for all $\mbf{d} \in (\F_{p^r}^+)^{n-1}$
\begin{align*}
&\psi_\mbf{b}((H\mbf{d}^T)^T) = \psi_\mbf{b}(\mbf{d})\\
&\Leftrightarrow \psi_\mbf{b}((H\mbf{d}^T)^T - \mbf{d}) = 1\\
&\Leftrightarrow \psi(\mbf{b}(H\mbf{d}^T - \mbf{d}^T)) = 1\\
&\Leftrightarrow \psi((\mbf{b}H - \mbf{b})\mbf{d}^T) = 1
\end{align*}

\subsection{Lemmas 5.6 and 5.7}\label{Ap5.67}
For any prime $p$, we define $$S(p,n) = \{\begin{pmatrix} A & 0_n\\
0_n & (A^{-1})^T\end{pmatrix}\begin{pmatrix} \text{Id}_n & B\\
0_n & \text{Id}_n\end{pmatrix} : A \in \Up_{n}(\F_{p^r}), B \in \sym(n,p^r)\}.$$ And it is easy to show that $S(p,n) \in \syl_p(Sp(2n,p^r))$ and that
$$S(p,n) \cong \sym(n,p^r) \rtimes \Up_{n}(\F_{p^r}),$$
where the action is given by $A(B) = ABA^T,$ where $B \in \sym(n,p^r), A \in \Up_n(\F_{p^r})$.

\begin{lma}[\ref{Lem5.6}] For $p \neq 2$, $S(p,n)$ the Sylow $p$-subgroup of $Sp(2n,p^r)$ defined above, $$Z(S(p,n)) = \{\begin{pmatrix} \text{Id}_n & D\\
0_n & \text{Id}_n\end{pmatrix} : D  = \begin{pmatrix} d & \mbf{0}\\
\mbf{0} & 0_{n-1} \end{pmatrix}\} \cong \F_{p^r}^+ \cong (\Z/p\Z)^r$$
\end{lma}

\noindent For the proof of of this Lemma, we need the following lemma:

\begin{lemma}\label{ADA} For $p \neq 2$, $D \in \sym(n,p^r)$, $AD = D(A^{-1})^T$
for all $A \in \Up_n(\F_{p^r})$ if and only if $D = \begin{pmatrix} d & \mbf{0}\\
\mbf{0} & 0_{n-1} \end{pmatrix}$.  \end{lemma}

\noindent Granting this lemma, we can calculate the center:

\begin{proof}
\begin{align*}
S(p,n) &= \{\begin{pmatrix} A & 0_n\\
0_n & (A^{-1})^T\end{pmatrix}\begin{pmatrix} \text{Id}_n & B\\
0_n & \text{Id}_n\end{pmatrix} : A \in \Up_n(\F_{p^r}), B \in \sym(n,p^r)\}\\
&= \{\begin{pmatrix} A & AB\\
0_n & (A^{-1})^T\end{pmatrix} : A \in \Up_n(\F_{p^r}), B \in \sym(n,p^r)\}.
\end{align*}
Note that
{\footnotesize \begin{align*}
\begin{pmatrix} A & AB\\
0_n & (A^{-1})^T\end{pmatrix}^{-1}\begin{pmatrix} C & CD\\
0_n & (C^{-1})^T\end{pmatrix}\begin{pmatrix} A & AB\\
0_n & (A^{-1})^T\end{pmatrix} = \begin{pmatrix} A^{-1}CA & A^{-1}CAB + A^{-1}CD(A^{-1})^T - B((A^{-1}CA)^{-1})^{T}\\
0_n & ((A^{-1}CA)^{-1})^{T}\end{pmatrix}
\end{align*}}
So $\begin{pmatrix} C & CD\\
0_n & (C^{-1})^T\end{pmatrix} \in Z(S(p,n))$ if and only if
$C \in Z(\Up_n(\F_{p^r}))$ and 
$$CD = CB + CA^{-1}D(A^{-1})^T - B(C^{-1})^{T}, \qquad \text{ for all } A \in \Up_n(\F_{p^r}),B \in \sym(n,p^r).$$
Choosing $A, B = \text{Id}_n$, we need $CD = C + CD - (C^{-1})^T$. So $C = (C^{-1})^T$ and thus $C = \text{Id}_n$.
So the other requirement above becomes 
$$D = A^{-1}D(A^{-1})^T \Leftrightarrow AD = D(A^{-1})^T, \qquad \text{ for all } A \in \Up_n(\F_{p^r}).$$
By Lemma \ref{ADA}, we get that
$$Z(S(p,n)) = \{\begin{pmatrix} \text{Id}_n & D\\
0_n & \text{Id}_n\end{pmatrix} : D = \begin{pmatrix} d & \mbf{0}\\
\mbf{0} & 0_{n-1} \end{pmatrix}\}$$

\end{proof}

\begin{proof}[Proof of Lemma \ref{ADA}]

\text{ } 

\noindent $\Leftarrow$: This is a straightforward calculation.

\noindent $\Rightarrow$: We will prove this by induction.

\textbf{Base Case}: When $n = 2$, we can write $A = \begin{pmatrix} 1 & a\\
0 & 1\end{pmatrix}$ and $D = \begin{pmatrix} x & y\\
y & z\end{pmatrix}$.

$$AD = \begin{pmatrix} x + ay & y + az\\
y & z\end{pmatrix},$$
and
$$D(A^{-1})^T = \begin{pmatrix} x - ay & y - az\\
y & z\end{pmatrix}.$$
So the condition that $AD = D(A^{-1})^T$ for all $A$ implies that $y = 0$ and $z = 0$.

\textbf{Induction Step:}
Write
$$D = \begin{pmatrix} d_{1,1} & d_{1,2} & d_{1,3} & \cdots & d_{1,n}\\
d_{1,2} & d_{2,2} & d_{2,3} & \cdots & d_{2,n}\\
\vdots &  & \ddots &  & \vdots\\
d_{1,n-1} & d_{2,n-1} & \cdots & d_{n-1,n-1} & d_{n-1,n}\\
d_{1,n} & d_{2,n} & \cdots & d_{n-1,n} & d_{n,n}\end{pmatrix}, \qquad A = \begin{pmatrix} 1 & 0 & 0 & \cdots & 0 \\
0 & 1 & 0 & \cdots & 0\\
& & \ddots & & \vdots\\
0 & 0 & \cdots & 1 & a_{n-1,n}\\
0 & 0 & 0 & \cdots & 1\end{pmatrix}.$$
Then
$$AD = \begin{pmatrix} d_{1,1} & \cdots & d_{1,n-1} & d_{1,n} \\
d_{2,2} & \cdots & d_{2,n-1} & d_{2,n}\\
\vdots & \ddots & & \vdots\\
d_{1,n-1} + a_{n-1,n}d_{1,n} & \cdots & d_{n-1,n-1} + a_{n-1,n}d_{n-1,n} & d_{n-1,n} + a_{n-1,n}d_{n,n}\\
d_{1,n} & \cdots & d_{n-1,n} & d_{n,n}\end{pmatrix}$$
and 
$$D(A^{-1})^T = \begin{pmatrix} d_{1,1} & \cdots & d_{1,n-2} & d_{1,n-1} - a_{n-1,n}d_{1,n} & d_{1,n}\\
d_{1,2} & \cdots & d_{2,n-2} & d_{2,n-1} - a_{n-1,n}d_{2,n} & d_{2,n}\\
\vdots & \ddots &  & \vdots\\
d_{1,n-1} & \cdots & d_{n-1,n-2} & d_{n-1,n-1} - a_{n-1,n}d_{n-1,n} & d_{n-1,n}\\
d_{1,n} & \cdots & d_{n,n-2} & d_{n-1,n} - a_{n-1,n}d_{n,n} & d_{n,n}\end{pmatrix}$$
In order for these to be equal for all $a_{n-1,n}$, we must have $d_{k,n} = 0$ for all $k$. So the matrix 
$$D' = \begin{pmatrix} d_{1,1} & d_{1,2} & d_{1,3} & \cdots & d_{1,n-1}\\
d_{1,2} & d_{2,2} & d_{2,3} & \cdots & d_{2,n-1}\\
\vdots &  & \ddots &  & \vdots\\
d_{1,n-2} & d_{2,n-2} & \cdots & d_{n-2,n-2} & d_{n-2,n}\\
d_{1,n-1} & d_{2,n-1} & \cdots & d_{n-2,n-1} & d_{n-1,n-1}\end{pmatrix}$$ satisfies the condition $A'D' = D'(A'^{-1})^T$ for all $A' \in \Up_{n-1}(\F_{p^r})$.
By induction, we conclude that $$D' = \begin{pmatrix} d & \mbf{0}\\
\mbf{0} & 0_{n-2}\end{pmatrix},$$ and hence $$D = \begin{pmatrix} d & \mbf{0}\\
\mbf{0} & 0_{n-1}\end{pmatrix}.$$

\end{proof}

\begin{lma}[\ref{Lem5.7}] For $S(2,n)$ the Sylow $p$-subgroup of $Sp(2n,2^r)$ defined above, \begin{align*}
Z(S(2,n)) &= \{\begin{pmatrix} \text{Id}_n & D\\
0_n & \text{Id}_n\end{pmatrix} : D_{i,j} = 0, \text{ for all } (i,j) \notin \{(1,1),(1,2),(2,1), D_{1,2} = D_{2,1}\} \cong (\F_{2^r}^+)^2 \cong (\Z/2\Z)^{2r}
\end{align*}
\end{lma}

\noindent For the proof, we need the following lemma:

\begin{lemma}\label{ADAPSP2n} For $p = 2$, $D \in \sym(n,2^r)$,  $AD = D(A^{-1})^T$
for all $A \in \Up_n(\F_{2^r})$ if and only if\\
 $D_{i,j} = 0, \text{ for all } (i,j) \notin \{(1,1),(1,2),(2,1)\}$.  \end{lemma}
 
\noindent Granting this lemma, we can calculate the center:
\begin{proof}
\begin{align*}
\syl_2(S(2,n)) &= \{\begin{pmatrix} A & AB\\
0_n & (A^{-1})^T\end{pmatrix} : A \in \Up_n(\F_{2^r}), B \in \sym(n,2^r)\}.
\end{align*}

\noindent Just as for $p \neq 2$, $\begin{pmatrix} C & CD\\
0_n & (C^{-1})^T\end{pmatrix} \in Z(\syl_p(PSp(n,2^r)))$ if and only if $C = \text{Id}_n$ and

$$D = A^{-1}D(A^{-1})^T \Leftrightarrow AD = D(A^{-1})^T, \qquad \text{ for all } A \in \Up_n(\F_{p^r}).$$
By Lemma \ref{ADAPSP2n}, then we have that

$$Z(S(2,n)) = \{\begin{pmatrix} \text{Id}_n & D\\
0_n & \text{Id}_n\end{pmatrix} : D_{i,j} = 0, \text{ for all } (i,j) \notin \{(1,1),(1,2),(2,1)\}\}$$

\end{proof}

\begin{proof}[Proof of Lemma \ref{ADAPSP2n}] \text{ }

\noindent $\Leftarrow$: This is a straightforward calculation.

\noindent $\Rightarrow$: We will prove this by induction.

\textbf{Base Case}: When $n = 2$, we can write $A = \begin{pmatrix} 1 & a\\
0 & 1\end{pmatrix}$ and $D = \begin{pmatrix} x & y\\
y & z\end{pmatrix}$.

$$AD = \begin{pmatrix} x + ay & y + az\\
y & z\end{pmatrix},$$
and
$$D(A^{-1})^T = \begin{pmatrix} x + ay & y\\
y+az & z\end{pmatrix}.$$
So the condition that $AD = D(A^{-1})^T$ for all $A$ implies that $z = 0$.

\begin{remark} This calculation is the key difference between odd and even characteristic. \end{remark}

\noindent \textbf{Induction Step:} Assume that $n > 2$. Write

$$D = \begin{pmatrix} d_{1,1} & d_{1,2} & d_{1,3} & \cdots & d_{1,n}\\
d_{1,2} & d_{2,2} & d_{2,3} & \cdots & d_{2,n}\\
\vdots &  & \ddots &  & \vdots\\
d_{1,n-1} & d_{2,n-1} & \cdots & d_{n-1,n-1} & d_{n-1,n}\\
d_{1,n} & d_{2,n} & \cdots & d_{n-1,n} & d_{n,n}\end{pmatrix}, \qquad A = \begin{pmatrix} 1 & 0 & 0 & \cdots & 0 \\
0 & 1 & 0 & \cdots & 0\\
& & \ddots & & \vdots\\
0 & 0 & \cdots & 1 & a_{n-1,n}\\
0 & 0 & 0 & \cdots & 1\end{pmatrix}.$$ 
Then
$$AD = \begin{pmatrix} d_{1,1} & \cdots & d_{1,n-1} & d_{1,n} \\
d_{1,2} & \cdots & d_{2,n-1} & d_{2,n}\\
\vdots & \ddots & & \vdots\\
d_{1,n-1} + a_{n-1,n}d_{1,n} & \cdots & d_{n-1,n-1} + a_{n-1,n}d_{n-1,n} & d_{n-1,n} + a_{n-1,n}d_{n,n}\\
d_{1,n} & \cdots & d_{n-1,n} & d_{n,n}\end{pmatrix}$$
and 
$$D(A^{-1})^T = \begin{pmatrix} d_{1,1} & \cdots & d_{1,n-2} & d_{1,n-1} + a_{n-1,n}d_{1,n} & d_{1,n}\\
d_{1,2} & \cdots & d_{2,n-2} & d_{2,n-1} + a_{n-1,n}d_{2,n} & d_{2,n}\\
\vdots & \ddots &  & \vdots\\
d_{1,n-1} & \cdots & d_{n-1,n-2} & d_{n-1,n-1} + a_{n-1,n}d_{n-1,n} & d_{n-1,n}\\
d_{1,n} & \cdots & d_{n,n-2} & d_{n-1,n} + a_{n-1,n}d_{n,n} & d_{n,n}\end{pmatrix}$$
In order for these to be equal for all $a_{n-1,n}$, we must have $d_{k,n} = 0$ for all $k$ except $k = n-1$.
Since $n > 2$, we can pick $$A = \begin{pmatrix} 1 & 0 & 0 & \cdots & 0 \\
0 & 1 & 0 & \cdots & 0\\
& & \ddots & & \vdots\\
0 & \cdots & 1 & a_{n-2,n-1} & 0\\
0 & 0 & \cdots & 1 & 0\\
0 & 0 & 0 & \cdots & 1\end{pmatrix}.$$
By comparing the entries of $AD$ and $D(A^{-1})^T$, we see that in order to have $AD = D(A^{-1})^T$ for all $a_{n-2,n-1}$, we must have $d_{k,n-1} = 0$ for all $k$ except $k = n-2$. In particular, we get that $d_{n,n-1} = d_{n-1,n} = 0$. Thus $d_{k,n} = 0$ for all $k$.  So the matrix 
$$D' = \begin{pmatrix} d_{1,1} & d_{1,2} & d_{1,3} & \cdots & d_{1,n-1}\\
d_{1,2} & d_{2,2} & d_{2,3} & \cdots & d_{2,n-1}\\
\vdots &  & \ddots &  & \vdots\\
d_{1,n-2} & d_{2,n-2} & \cdots & d_{n-2,n-2} & d_{n-2,n}\\
d_{1,n-1} & d_{2,n-1} & \cdots & d_{n-2,n-1} & d_{n-1,n-1}\end{pmatrix}$$ satisfies the condition $A'D' = D'(A'^{-1})^T$ for all $A' \in \Up_{n-1}(\F_{p^r})$. By induction, we conclude that 
$$D_{i,j} = 0, \text{ for all } (i,j) \notin \{(1,1),(1,2),(2,1)\}.$$

\end{proof}

\subsection{Section 5.3 Calculation}\label{5.3calc}
The calculation that $H \in L_{\mbf{b}}$ if and only if $\psi(\mbf{b} \cdot (\mbf{hdh^T} - \mbf{d})) = 1$ for all $\mbf{d} \in (\F_{p^r})^{n(n+1)/2},$ where $\mbf{hdh^T}$ is the vector corresponding to $HDH^T$ under the isomorphism $\sym(n,p^r) \cong (\F_{p^r}^+)^{n(n+1)/2}$:

\begin{remark} In all of the following, we view $\psi_{\mbf{b}}$ as a map on $\Delta \cong \sym(n,p^r) \cong \F_{p^r}^{n(n+1)/2}$. So $\psi_{\mbf{b}}(D,\text{Id}) = \psi_{\mbf{b}}(D) = \psi(\mbf{b} \cdot \mbf{d})$, where $\mbf{d}$ is the vector corresponding to the matrix $D$. \end{remark}

\noindent Note that $(0_n,H) \in L_\mbf{b}$ if and only if for all $\mbf{d} \in (\F_{p^r})^{n(n+1)/2} = D \in \sym(n,p^r)$,
$$\psi_{\mbf{b}}((0_n,H)(D,\text{Id}_n)(0_n,H^{-1})) = \psi_{\mbf{b}}(D, \text{Id}_n).$$
Let $\mbf{hdh^T}$ denote the vector corresponding to $HDH^T$. Then since
$$\psi_{\mbf{b}}((0_n,H)(D,\text{Id}_n)(0_n,H^{-1})) = \psi(\mbf{b} \cdot \mbf{hdh^T}),$$
and
$$\psi_{\mbf{b}}(D,\text{Id}_n) = \psi(\mbf{b} \cdot \mbf{d}),$$
we get that $(0_n,H) \in L_\mbf{b}$ if and only if for all $\mbf{d} \in (\F_{p^r})^{n(n+1)/2} = D \in \sym(n,p^r)$,
$$\psi(\mbf{b} \cdot (\mbf{hdh^T} - \mbf{d})) = 1.$$

\subsection{Proposition 5.9}\label{Ap5.9}
\begin{prop}[\ref{Ls5}] For $p = 2$, $n = 2$, 
$$\min_{\mbf{b} \in (\F_{p^r}^+)^{3}, \text{ } b_1 \neq 0, b_2\neq 0}\dim(\theta_{\mbf{b},1})  = 2^{r-1}.$$ 
This minimum is achieved when $\mbf{b} = (b_1,b_2,0)$ with $b_1 \neq 0, b_2 \neq 0$.\\
If $\mbf{b} = (b_1, b_2, 0)$ with $b_1 \neq 0, b_2 \neq 0$, then 
$$\dim(\theta_{\mbf{b},1})  = 2^r.$$\end{prop}

\begin{proof} 

We will prove the proposition in two steps:

\textbf{Step 1: Proving that for } $\mbf{p = 2, n = 2, s = (b_i), (b_1,b_2) \neq (0,0)}:$ \textbf{ if } $\mbf{b_1, b_2 \neq 0}$, \textbf{ then } $\mbf{|L_\mbf{b}| \leq 2}$, \textbf{ and otherwise } $\mbf{|L_\mbf{b}| = 1}$.

$$HDH^T - D = \begin{pmatrix} [\sum_{l = 1}^2 (h_{1,l}\sum_{k=1}^2 d_{l,k}h_{1,k})] - d_{1,1} & (\sum_{k=1}^2 d_{k,2}h_{1,k}) - d_{1,2}\\
(\sum_{l=1}^2 h_{1,l}d_{l,2}) - d_{1,2} & 0 \end{pmatrix}$$
Let $p = 2$, $\mbf{b} = (b_i)$ with $(b_1,b_2) \neq (0,0)$.

\begin{calc1} Choose $d_{i,j} = 0$ except for $d_{2,2}$. \end{calc1}

\noindent Then we get that 

$$HDH^T - D = \begin{pmatrix} h_{1,2}^2d_{2,2} & h_{1,2}d_{2,2}\\
h_{1,2}d_{2,2} & 0 \end{pmatrix}$$
Thus we have 
\begin{align*}
\mbf{b} \cdot (\mbf{hdh^T} - \mbf{d}) &= B_{1,1}h_{1,2}^2d_{2,2} + B_{1,2}h_{1,2}d_{2,2}\\
&= d_{2,2}h_{1,2}(B_{1,1}h_{1,2} + B_{1,2})
\end{align*}
Then since $\psi$ is non-trivial, we must have $h_{1,2}(B_{1,1}h_{1,2} + B_{1,2}) = 0$.
Thus either $h_{1,2} = 0$ or $B_{1,1}h_{1,2} + B_{1,2} = 0$. If $B_{1,1} \neq 0$, $B_{1,2} \neq 0$, then either $h_{1,2} = 0$ or $h_{1,2} = \frac{B_{1,2}}{B_{1,1}}$. If $B_{1,1} \neq 0$, $B_{1,2} = 0$ or $B_{1,1} = 0$, $B_{1,2} \neq 0$, then $h_{1,2} = 0$. Our findings can be summarized in a chart as follows (we only care when $(B_{1,1},B_{1,2}) \neq (0,0)$):

\begin{tabular}{c|c|c}
Case: & result & options\\
\hline
$B_{1,1} \neq 0, B_{1,2} \neq 0$ & $h_{1,2} = 0$ or $h_{1,2} = \frac{B_{1,2}}{B_{1,1}}$ & 2\\
$B_{1,1} \neq 0, B_{1,2} = 0$ & $h_{1,2} = 0$ & 1\\
$B_{1,1} = 0, B_{1,2} \neq 0$ & $h_{1,2} = 0$ & 1\\
\end{tabular}

\bigskip

\noindent Thus we can conclude that for all $s = (b_i)$ with $(b_1,b_2) \neq (0,0)$, then for $b_1,b_2\neq 0$, $|L_\mbf{b}| \leq 2$ and otherwise $|L_\mbf{b}| = 1$.

\textbf{Step 2: Showing that when} $\mbf{s = (b_1,b_2,0)}$ \textbf{with} $\mbf{b_1 \neq 0,b_2 \neq 0, |L_\mbf{b}| = 2}$.

\noindent For $\mbf{b} = (b_1,b_2,b_3)$,

\begin{align*}
\mbf{b} \cdot (\mbf{hdh^T} - \mbf{d}) &= b_1\left([\sum_{l = 1}^2 (h_{1,l}\sum_{k=1}^2 d_{l,k}h_{1,k})] - d_{1,1}\right) + b_2\left([\sum_{k=1}^2 d_{k,2}h_{1,k}] - d_{1,2}\right)\\
&= b_1h_{1,2}^2d_{2,2} + b_2d_{2,2}h_{1,2} \qquad \text{ since we are working in char 2}\\
&= d_{2,2}h_{1,2}(b_1h_1 + b_2)
\end{align*}
If $b_1 \neq 0, b_2 \neq 0$, then either $h_{1,2} = 0$ or $h_1 = \frac{b_2}{b_1}$.  In either case, the above is identically zero. Thus $|L_\mbf{b}| = 2$.

\end{proof}

\subsection{Proposition 5.10}\label{Ap5.10}

\begin{prop}[\ref{Ls2}] For $p = 2$, $n > 2$,
$$\min_{\mbf{b} \in (\F_{p^r}^+)^{n(n+1)/2}, \text{ } b_2 \neq 0}\dim(\theta_{\mbf{b},1}) =  2^{r(2n-3) -1}.$$
This minimum is achieved when $\mbf{b} = (b_i) = (b_1,b_2,0,\ldots, 0)$ with $b_1, b_2 \neq 0$. 
$$\min_{\mbf{b} \in (\F_{p^r}^+)^{n(n+1)/2}, \text{ } b_1 \neq 0}\dim(\theta_{\mbf{b},1})  =  2^{r(n-1) - 1}.$$
This minimum is achieved when $\mbf{b} = (b_i) = (b_1,0,b_3, \ldots, 0)$ with $b_1,b_3 \neq 0$.\end{prop}

\begin{proof}

\noindent Again, we will prove this in two steps:

\textbf{Step 1: Proving that for } $\mbf{ p = 2, n > 2, s = (b_i), (b_1,b_2) \neq (0,0)}$:
\textbf{ If } $\mbf{b_2 \neq 0}$, \textbf{ then }\\
$\mbf{|L_\mbf{b}| \leq 2^{r(n-2)(n-3)/2 + 1}}$, \textbf{and if } $\mbf{b_2 = 0 (\Rightarrow b_1 \neq 0)}$, \textbf{ then } $\mbf{|L_\mbf{b}| \leq  2^{r(n-1)(n-2)/2+1}}.$

\begin{calc} For $j_0 > 2$, choose $d_{i,j} = 0$ except for $d_{1,j_0} = d_{j_0,1}$. \end{calc} 

\noindent Then $$\mbf{b} \cdot (\mbf{hdh^T} - \mbf{d}) =  \sum_{i=2}^{j_0-1} h_{i,j_0}d_{1,j_0}B_{1,i} = d_{1,j_0}\sum_{i=2}^{j_0-1} h_{i,j_0}B_{1,i}$$
So for all $j_0 > 2$, we must have
$$\sum_{i=2}^{j_0-1} h_{i,j_0}B_{1,i} = 0.$$
For $j_0 = 3$, this gives $h_{2,j_0}B_{1,2} = 0,$ and thus if $B_{1,2} \neq 0$, we must have $h_{2,j_0} = 0$. For $2 \leq k \leq n$, if $B_{1,k} \neq 0$, then for all $j_0 > 3$, given $h_{i,j_0}$ for $i \neq 1,k$, the above dictates $h_{k,j_0}$: 
$$h_{k,j_0} = \frac{-1}{B_{1,k}} \sum_{i=2, i \neq k}^{j_0-1} h_{i,j_0}B_{1,i}.$$

\begin{calc} Now for $j_0 > 1$, choose $d_{i,j} = 0$ except for $d_{j_0,j_0}$. \end{calc}

\noindent Then 
$$\mbf{b} \cdot (\mbf{hdh^T} - \mbf{d}) = d_{j_0,j_0} \left(\sum_{l=1}^{j_0-1}\sum_{k=l}^{j_0} B_{l,k}h_{l,j_0}h_{k,j_0}\right)$$
So for all $j_0 \neq 1$, we must have 
$$\sum_{l=1}^{j_0-1}\sum_{k=l}^{j_0} B_{l,k}h_{l,j_0}h_{k,j_0} = 0.$$
Thus we have that for all $j_0 \neq 1$,
$$h_{1,j_0}(\sum_{k=1}^{j_0} B_{1,k}h_{k,j_0}) + \sum_{l=2}^{j_0-1}\sum_{k=l}^{j_0} B_{l,k}h_{l,j_0}h_{k,j_0} = 0$$
For $j_0 = 2$, this tells us $0 = h_{1,2}(B_{1,1}h_{1,2} + B_{1,2})$. If $B_{1,2} = 0 (\Rightarrow B_{1,1} \neq 0)$ or $B_{1,1} = 0 (\Rightarrow B_{1,2} \neq 0)$, then this implies that $h_{1,2} = 0$. If $B_{1,2} \neq 0$ and $B_{1,1} \neq 0$, then we have two options for $h_{1,2}$: $h_{1,2} = 0$ and $h_{1,2} = \frac{B_{1,2}}{B_{1,1}}$. For $j_0 > 2$, this is a quadratic expression for $h_{1,j_0}$ in terms of $B_{i,j}$ and $h_{k,j_0}$ for $k > 1$, namely
$$B_{1,1}h_{1,j_0}^2 + (\sum_{k=2}^{j_0} B_{1,k}h_{k,j_0})h_{1,j_0} + \sum_{l=2}^{j_0-1}\sum_{k=l}^{j_0} B_{l,k}h_{l,j_0}h_{k,j_0} = 0$$ 
Thus for $j_0 > 2$, given $h_{i,j_0}$ for $i > 1$, there are up to two options for $h_{1,j_0}$.

\begin{calc} Now for $j_0 > 2$, choose $d_{i,j} = 0$ except for $d_{2,j_0} = d_{j_0,2}$. \end{calc}

\noindent Then 
$$\mbf{b} \cdot (\mbf{hdh^T} - \mbf{d}) = d_{2,j_0}\left(B_{1,2}h_{1,j_0} + \sum_{i=2}^{j_0} B_{1,i}h_{i,j_0}h_{1,2} + \sum_{i=3}^{j_0-1} B_{2,i}h_{i,j_0}\right)$$
So for all $j_0 > 2$, we must have 
$$B_{1,2}h_{1,j_0} + \sum_{i=2}^{j_0} B_{1,i}h_{i,j_0}h_{1,2} + \sum_{i=3}^{j_0-1} B_{2,i}h_{i,j_0} = 0.$$
If $B_{1,2} \neq 0$, then for all $j_0 > 2$, given $h_{i,j_0}$ for $i > 2$, the above dictates $h_{1,j_0}$: 
$$h_{1,j_0} = \frac{-1}{B_{1,2}} \left(\sum_{i=2}^{j_0} B_{1,i}h_{i,j_0}h_{1,2} + \sum_{i=3}^{j_0-1} B_{2,k}h_{i,j_0}\right)$$

\bigskip

\begin{case''} $b_2 \neq 0$ \end{case''}

\noindent If $B_{1,2} = b_2 \neq 0$, then we have from the first calculation that for all $j_0 > 1$, given $h_{i,j_0}$ for $i > 2$, $h_{2,j_0}$ are dictated. By the second calculation we have that there are at most two options for $h_{1,2}$.   And by the third calculation, $h_{1,j_0}$ is dictated for $j_0 > 2$. Thus for $b_2 \neq 0$, we can conclude that 
\begin{align*}
|L_\mbf{b}| &\leq |\{H : \text{ two options for } H_{1,2},  \text{ and } \forall j > 2, H_{1,j}, H_{2,j} \text{ fixed},\}|\\
&= 2|\Up_{n-2}(\F_{p^r})|\\
&= 2^{r(n-2)(n-3)/2 + 1}.
\end{align*}

\begin{case''} $b_2 = 0, b_3 \neq 0$ \end{case''}

\noindent If $B_{1,2} = b_2 = 0 (\Rightarrow B_{1,1} \neq 0)$: We have by the second calculation that for $j_0 > 2$,

$$0 = B_{1,1}h_{1,j_0}^2 + (\sum_{k=3}^{j_0} B_{1,k}h_{k,j_0})h_{1,j_0} + \sum_{l=2}^{j_0-1}\sum_{k=l}^{j_0} B_{l,k}h_{l,j_0}h_{k,j_0}$$
For $j_0 = 2$, we get $B_{1,1}h_{1,2}^2 = 0$. Thus we must have $h_{1,2} = 0$.
For $j_0 = 3$, we get $0 = B_{1,1}h_{1,3}^2 + B_{1,3}h_{1,3} = h_{1,3}(B_{1,1}h_{1,3} + B_{1,3})$. Thus either $h_{1,3} = 0$ or $h_{1,3} = \frac{B_{1,3}}{B_{1,1}}$.
For $j_0 > 3$,  we have from the first calculation that$\sum_{i=3}^{j_0-1} h_{i,j_0}B_{1,i} = 0$, so the equality from the second calculation becomes
$$0 = B_{1,1}h_{1,j_0}^2 + B_{1,j_0}h_{1,j_0} + \sum_{l=2}^{j_0-1}\sum_{k=l}^{j_0} B_{l,k}h_{l,j_0}h_{k,j_0}$$
We will use the following proposition:
\begin{proposition}[\cite{Po}, Proposition 1]\label{pom1} In a finite field of order $2^r$, for $f(x) = ax^2 + bx + c$, we have have the following:
\begin{enumerate}[(i)]
\item $f$ has exactly one root $\Leftrightarrow$ $b = 0$.
\item $f$ has exactly two roots $\Leftrightarrow$ $b \neq 0$ and $\text{Tr}(\frac{ac}{b^2}) = 0$.
\item $f$ has no root $\Leftrightarrow$ $b \neq 0$ and $\text{Tr}(\frac{ac}{b^2}) = 1,$
\end{enumerate}
where $\text{Tr}(x) = x + x^2 + \cdots + x^{2^r - 1}$.
\end{proposition}

\noindent So, for $j_0 > 3$, if $B_{1,j_0} = 0$, then there is only one option for $h_{1,j_0}$. Otherwise, it might have two options or no options. Thus we have the following for $j_0 > 3$:  If $B_{1,j_0} = 0$, then there is one option for $h_{1,j_0}$, but $h_{k,j_0}$ can be anything for $k > 1$. And if $B_{1,j_0} \neq 0$, then there is only one option for $h_{j_0,k_0}$ for all $k_0 > 2$ (by the first calculation with $k = j_0, j_0 = k_0$), but $h_{1,j_0}$ might have two options. So we can obtain an upper bound for $L_\mbf{b}$ by choosing $B_{1,j} = 0$ for all $j > 3$ and assuming all the options are in $L_\mbf{b}$. In this case $h_{2,j}$ can be anything, but $h_{1,j}$ is fixed for all $j$ except $j = 3$, and there are two options for $h_{1,3}$ So we get that 
\begin{align*}
|L_\mbf{b}| &\leq |\{H : H_{1,j} \text{ fixed } \forall j \neq 3, H_{1,3} = 0 \text{ or } \frac{B_{1,3}}{B_{1,1}}\}|\\
&= 2|\Up_{n-1}(\F_{2^r})|\\
&= 2^{r(n-1)(n-2)/2 + 1}
\end{align*}

\textbf{Step 2: Showing that for} $\mbf{p = 2, n > 2}$: \textbf{When} $\mbf{b = (b_1,b_2,0, \cdots, 0)}$ \textbf{with} $\mbf{b_1,b_2 \neq 0,}$\\
 $\mbf{|L_\mbf{b}| = 2^{r(n-2)(n-3)/2 + 1},}$
\textbf{and when} $\mbf{b = (b_1,0,b_3,\cdots,0)}$ \textbf{with} $\mbf{ b_1,b_3 \neq 0, |L_\mbf{b}| = 2^{r(n-1)(n-2)/2 + 1}.}$

\smallskip

\noindent Let $p = 2$, $\mbf{b} = (b_1,b_2, \cdots, b_n, 0, \cdots,0)$. And let $B$ be the corresponding matrix. Then
$$\mbf{b} \cdot (\mbf{hdh^T} - \mbf{d}) = b_1\sum_{k=2}^n h_{1,k}^2d_{k,k} + \sum_{j=2}^n b_j\left([\sum_{l=j}^n (h_{j,l}\sum_{k=1}^n d_{l,k}h_{1,k})] - d_{1,j}\right)$$

\begin{case'} $b_1, b_2 \neq 0, b_3,\cdots, b_n = 0$. \end{case'}

\noindent Since $B_{1,2} = b_2 \neq 0$, then we have from the first calculation in Step 1 that for all $j_0 > 2$,  
$$h_{2,j_0} = \frac{-1}{B_{1,2}} \sum_{i=3}^{j_0-1} h_{i,j_0}B_{1,i} = 0.$$
By the second calculation we have that there are two options for $h_{1,2}$: $h_{1,2} = 0$ and $h_{1,2} = \frac{B_{1,2}}{B_{1,1}}$
And by the third calculation, for $j_0 > 2$,
\begin{align*}
h_{1,j_0} = \frac{-1}{B_{1,2}} \left(\sum_{i=2}^{j_0} B_{1,i}h_{i,j_0}h_{1,2} + \sum_{i=3}^{j_0-1} B_{2,k}h_{i,j_0}\right) = \frac{-1}{B_{1,2}} B_{1,2}h_{2,j_0}h_{1,2} = 0
\end{align*}
Thus we have 
$$\mbf{b} \cdot (\mbf{hdh^T} - \mbf{d}) = d_{2,2}h_{1,2}(B_{1,1}h_{1,2} + B_{1,2})$$
So whether $h_{1,2} = 0$ or $h_{1,2} = \frac{B_{1,2}}{B_{1,1}}$, this is identically 0. Therefore 
\begin{align*} 
|L_\mbf{b}| = |\{H : H_{1,2} = 0 \text{ or } H_{1,2} = \frac{B_{1,2}}{B_{1,1}}, H_{1,j} = 0 = H_{2,j} \text{ } \forall j > 0\}| = 2|\Up_{n-2}(\F_{2^r})| = 2^{r(n-2)(n-3)/2 + 1}
\end{align*}

\begin{case'} $b_1 \neq 0, b_2 = \cdots = b_n = 0$. \end{case'}

\noindent If $B_{1,k} = b_k = 0$ for $2 \leq k \leq n$: We have the following by the work in Step 1:

$h_{1,2} = 0$. By the second calculation we have that there are two options for $h_{1,3}$: $h_{1,2} = 0$ and $h_{1,3} = \frac{B_{1,3}}{B_{1,1}}$. And for $j_0 > 3$,
$$0 = B_{1,1}h_{1,j_0}^2 + B_{1,j_0}h_{1,j_0} + \sum_{l=2}^{j_0-1}\sum_{k=l}^{j_0} B_{l,k}h_{l,j_0}h_{k,j_0} = B_{1,1}h_{1,j_0}^2$$
So we have $h_{1,j_0} = 0$ for $j_0 \neq 1$. Thus
\begin{align*}
\mbf{b} \cdot (\mbf{hdh^T} - \mbf{d}) &= b_1\sum_{k=2}^n h_{1,k}^2d_{k,k} &\text{ since } b_i = 0 \text{ for } i > 1\\
&= 0  &\text{ since } h_{1,j_0} = 0 \text{ for } j_0 \neq 1\\
\end{align*}
Therefore
$$|L_\mbf{b}| = |\{H : H_{1,3} = 0 \text{ or } H_{1,3} = \frac{B_{1,3}}{B_{1,1}}, H_{1,j_0} = 0 \text{ for } j_0 \neq 1,3\}| = 2|\Up_{n-1}(\F_{2^r})| = 2^{r(n-1)(n-2)/2+1}$$

\end{proof}

\subsection{Lemma \ref{O-syl}}\label{Ap5}

\begin{lma}[Lemma \ref{O-syl}]
The Sylow $p$-subgroups of $O^+(2m,p^r)$ and $O^-(2m,p^r)$ are isomorphic.  \end{lma}

\begin{proof}

Let
{\small \begin{align*} S^-(p,2m) &= \{ \begin{pmatrix} A & A\mbf{y}^T & A\mbf{z}^T & AB\\
\mbf{0} & 1 & 0 & -\mbf{y}\\
\mbf{0} & 0 & 1 & \eta^{-1}\mbf{z}\\
0_m & \mbf{0} & \mbf{0} & (A^{-1})^T \end{pmatrix} : \mbf{y}, \mbf{z} \in \F_{p^r}^{m-1}, A \in \Up_{m-1}(\F_{p^r}), B + B^T = \eta^{-1}\mbf{z}^T\mbf{z}-\mbf{y}^T\mbf{y}\}.\end{align*}} Note that 
$$|S^-(p,2m)| = (p^{r})^{(m-1)(m-2)/2} \cdot (p^r)^{2(m-1)} \cdot (p^r)^{(m-1)(m-2)/2} = p^{rm(m-1)} = |O^-(2m,p^r)|_p.$$
Also, $S^-(p,2m) \subset O(2m+1,p^r)$ since
\begin{align*}
&\begin{pmatrix} A^T & \mbf{0} & \mbf{0} & 0_{m-1}\\
\mbf{y}A^T & 1 & 0 & \mbf{0}\\
\mbf{z}A^T & 0 & 1 & \mbf{0}\\
B^TA^T & -\mbf{y}^T & \eta^{-1}\mbf{z}^T & A^{-1} \end{pmatrix}\begin{pmatrix} 0_{m-1} & \mbf{0} & \mbf{0} & \text{Id}_{m-1}\\
\mbf{0} & 1 & 0 & \mbf{0}\\
\mbf{0} & 0 & -\eta & \mbf{0}\\
\text{Id}_{m-1} & \mbf{0} & \mbf{0} & 0_{m-1}\end{pmatrix}\begin{pmatrix} A & A\mbf{y}^T & A\mbf{z}^T & AB\\
\mbf{0} & 1  & 0 & -\mbf{y}\\
\mbf{0} & 0 & 1 & \eta^{-1}\mbf{z}\\
0_{m-1} & \mbf{0} & \mbf{0} & (A^{-1})^T \end{pmatrix}\\
&= \begin{pmatrix} 0_{m-1} & \mbf{0}  & \mbf{0} & A^T \\
\mbf{0} & 1 & 0 & \mbf{y}A^T\\
\mbf{0} & 0 & -\eta & \mbf{z}A^T\\
A^{-1} & -\mbf{y}^T & -\mbf{z}^T & B^TA^T  \end{pmatrix}\begin{pmatrix} A & A\mbf{y}^T & A\mbf{z}^T & AB\\
\mbf{0} & 1  & 0 & -\mbf{y}\\
\mbf{0} & 0 & 1 & \eta^{-1}\mbf{z}\\
0_{m-1} & \mbf{0} & \mbf{0} & (A^{-1})^T \end{pmatrix}\\
&=\begin{pmatrix} 0_{m-1} & \mbf{0} & \mbf{0} & \text{Id}_{m-1}\\
\mbf{0} & 1 & 0 & \mbf{0}\\
\mbf{0} & 0 & -\eta & \mbf{0}\\
\text{Id}_{m-1} & \mbf{0} & \mbf{0} & B + \mbf{y}^T\mbf{y} - \eta^{-1}\mbf{z}^T\mbf{z} + B^T \end{pmatrix}\\
&= \begin{pmatrix} 0_{m-1} & \mbf{0} & \mbf{0} & \text{Id}_{m-1}\\
\mbf{0} & 1 & 0 & \mbf{0}\\
\mbf{0} & 0 & -\eta & \mbf{0}\\
\text{Id}_{m-1} & \mbf{0} & \mbf{0} & 0_{m-1}\end{pmatrix}
\end{align*}

Furthermore, for $$\begin{pmatrix} A & A\mbf{y}^T & A\mbf{z}^T & AB\\
\mbf{0} & 1 & 0 & -\mbf{y}\\
\mbf{0} & 0 & 1 & \eta^{-1}\mbf{z}\\
0_{m-1} & \mbf{0} & \mbf{0} & (A^{-1})^T \end{pmatrix}, \begin{pmatrix} A' & A'\mbf{y'}^T & A'\mbf{z'}^T & A'B'\\
\mbf{0} & 1 & 0 & -\mbf{y'}\\
\mbf{0} & 0 & 1 & \eta^{-1}\mbf{z'}\\
0_{m-1} & \mbf{0} & \mbf{0} & (A'^{-1})^T \end{pmatrix}\in S^-(p,2m),$$
we have
\begin{align*}&\begin{pmatrix} A & A\mbf{y}^T & A\mbf{z}^T & AB\\
\mbf{0} & 1 & 0 & -\mbf{y}\\
\mbf{0} & 0 & 1 & \eta^{-1}\mbf{z}\\
0_{m-1} & \mbf{0} & \mbf{0} & (A^{-1})^T \end{pmatrix}, \begin{pmatrix} A' & A'\mbf{y'}^T & A'\mbf{z'}^T & A'B'\\
\mbf{0} & 1 & 0 & -\mbf{y'}\\
\mbf{0} & 0 & 1 & \eta^{-1}\mbf{z'}\\
0_{m-1} & \mbf{0} & \mbf{0} & (A'^{-1})^T \end{pmatrix}\\
&=\begin{pmatrix} AA' & AA'\mbf{y'}^T + A\mbf{y}^T  & AA
\mbf{z'}^T + A\mbf{z}^T & AA'B' - A \mbf{y}^T\mbf{y'} + \eta^{-1}A\mbf{z}^T\mbf{z'} + AB(A'^{-1})^T\\
\mbf{0} & 1  & 0 &  -\mbf{y'} - \mbf{y}((A')^{-1})^T\\
\mbf{0} & 0 & 1 & \eta^{-1}\mbf{z'} + \eta^{-1}\mbf{z}((A')^{-1})^T\\
0_{m-1} & \mbf{0} & \mbf{0} & ((AA')^{-1})^T  \end{pmatrix}.
\end{align*}

Note that $$AA'(\mbf{y'} + \mbf{y}((A')^{-1})^T)^T = AA'\mbf{y'}^T + A\mbf{y}^T,$$
$$AA'(\mbf{z'} + \mbf{z}((A')^{-1})^T)^T = AA'\mbf{z'}^T + A\mbf{z}^T,$$
$$AA'B' - A \mbf{y}^T\mbf{y'} + \eta^{-1}A\mbf{z}^T\mbf{z'} + AB((A')^{-1})^T = AA'(B' - (A')^{-1}\mbf{y}^T\mbf{y'} + \eta^{-1}(A')^{-1}\mbf{z}^T\mbf{z'}  + (A')^{-1}B((A')^{-1})^T),$$
and
\begin{align*}
&(B' - (A')^{-1}\mbf{y}^T\mbf{y'}+ \eta^{-1}(A')^{-1}\mbf{z}^T\mbf{z'} + (A')^{-1}B((A')^{-1})^T)\\
&\qquad  + (B' - (A')^{-1}\mbf{y}^T\mbf{y'}+ \eta^{-1}(A')^{-1}\mbf{z}^T\mbf{z'} + (A')^{-1}B((A')^{-1})^T)^T\\ 
&= (B' + B'^T) - ((A')^{-1}\mbf{y}^T\mbf{y'} + \mbf{y'}^T\mbf{y}(A'^{-1})^T) + \eta^{-1}((A')^{-1}\mbf{z}^T\mbf{z'} + \mbf{z'}^T\mbf{z}(A'^{-1})^T)\\
&\qquad + ((A')^{-1}B((A')^{-1})^T + (A')^{-1}B^T((A')^{-1})^T)\\
&= \eta^{-1}\mbf{z}^T\mbf{z}-\mbf{y'}^T\mbf{y'} - ((A')^{-1}\mbf{y}^T\mbf{y'} + \mbf{y'}^T\mbf{y}((A')^{-1})^T) + \eta^{-1}((A')^{-1}\mbf{z}^T\mbf{z'} + \mbf{z'}^T\mbf{z}((A')^{-1})^T)\\
&\qquad + (A')^{-1}(\eta^{-1}\mbf{z}^T\mbf{z}-\mbf{y}^T\mbf{y})((A')^{-1})^T
\end{align*}
while
\begin{align*}
&\eta^{-1}(\mbf{z'} + \mbf{z}((A')^{-1})^T)^T(\mbf{z'} + \mbf{z}((A')^{-1})^T) -(\mbf{y'} + \mbf{y}((A')^{-1})^T)^T(\mbf{y'} + \mbf{y}((A')^{-1})^T)\\
&= \eta^{-1}(\mbf{z'}^T + (A')^{-1}\mbf{z}^T)(\mbf{z'} + \mbf{z}((A')^{-1})^T) + (-\mbf{y'}^T - (A')^{-1}\mbf{y}^T)(\mbf{y'} + \mbf{y}((A')^{-1})^T)\\
&= \eta^{-1}\mbf{z'}^T\mbf{z'} + \eta^{-1}\mbf{z'}^T\mbf{z}((A')^{-1})^T + \eta^{-1}(A')^{-1}\mbf{z}^T\mbf{z'} + \eta^{-1}(A')^{-1}(-\mbf{z}^T\mbf{z})((A')^{-1})^T\\
&\qquad -\mbf{y'}^T\mbf{y'} - \mbf{y'}^T\mbf{y}((A')^{-1})^T - (A')^{-1}\mbf{y}^T\mbf{y'} - (A')^{-1}(-\mbf{y}^T\mbf{y})((A')^{-1})^T
\end{align*}

Thus the product is in $S^-(p,2m)$. So $S^-(p,2m)$ is a subgroup of $O^-(2m,p^r)$ of the desired order.  Hence $S^-(p,2m) \in \syl_p(O^-(2m,p^r))$.

If I choose $$A' = \begin{pmatrix} 0_{m-1} & \mbf{0} & \mbf{0} & \text{Id}_{m-1}\\
\mbf{0} & 1 & 0 & \mbf{0}\\
\mbf{0} & 0 & -1 & \mbf{0}\\
\text{Id}_{m-1} & \mbf{0} & \mbf{0} &  0_{m-1}\end{pmatrix}$$
and define 
$$O'(2m,p^r) := \{M \in GL(2m,\F_{p^r}) : M^TAM = A'\},$$
Let $$A^+ = \begin{pmatrix} 0_m & \text{Id}_m\\
\text{Id}_m & 0_m\end{pmatrix}$$ and define 
$$O^+(2m,p^r) := \{M \in GL(2m,\F_{p^r}) : M^T A^+ M = A^+\}.$$ 
Let $X_1 = \begin{pmatrix} \text{Id}_{m-1} & & \\
 & \frac{1}{2} & 1 & \\
 & \frac{1}{2}  & -1 & \\
 & & & \text{Id}_{m-1} \end{pmatrix}$ and $X_2 = \begin{pmatrix} \text{Id}_m &  & \\
 & \mbf{0} & 1\\
 & \text{Id}_{m-1} & \mbf{0} \end{pmatrix}$. Let $B = X_1X_2$. Then
\begin{align*} B^TA'B 
 &= X_2^T\begin{pmatrix} \text{Id}_{m-1} & & \\
 & \frac{1}{2} & \frac{1}{2} & \\
 & 1 & -1 & \\
 & & & \text{Id}_{m-1} \end{pmatrix}\begin{pmatrix}  &  &  & \text{Id}_{m-1}\\
 & 1 & 0 & \\
 & 0 & -1 & \\
\text{Id}_{m-1} &  &  &  \end{pmatrix}\begin{pmatrix} \text{Id}_{m-1} & & \\
 & \frac{1}{2} & 1 & \\
 & \frac{1}{2}  & -1 & \\
 & & & \text{Id}_{m-1} \end{pmatrix}X_2\\
 &= X_2^T\begin{pmatrix}  & & & \text{Id}_{m-1}\\
 & \frac{1}{2} & -\frac{1}{2}& \\
 & 1  & 1 & \\
\text{Id}_{m-1} & & &  \end{pmatrix}\begin{pmatrix} \text{Id}_{m-1} & & \\
 & \frac{1}{2} & 1 & \\
 & \frac{1}{2}  & -1 & \\
 & & & \text{Id}_{m-1} \end{pmatrix}X_2\\
&= \begin{pmatrix} \text{Id}_m &  & \\
 & \mbf{0} & \text{Id}_{m-1} \\
 & 1 & \mbf{0} \end{pmatrix}\begin{pmatrix}  &  &  & \text{Id}_{m-1}\\
 & 0 & 1 & \\
 & 1 & 0 & \\
\text{Id}_{m-1} &  &  &  \end{pmatrix}\begin{pmatrix} \text{Id}_m &  & \\
 & \mbf{0} & 1\\
 & \text{Id}_{m-1} & \mbf{0} \end{pmatrix}\\
 &= \begin{pmatrix} & & & \text{Id}_{m-1}\\
 & & 1 & \\
 \text{Id}_{m-1} & & & \\
    & 1 & & \end{pmatrix}\begin{pmatrix} \text{Id}_m &  & \\
 & \mbf{0} & 1\\
 & \text{Id}_{m-1} & \mbf{0} \end{pmatrix}\\
 &= \begin{pmatrix} & \text{Id}_m\\
 \text{Id}_m & \end{pmatrix}\\
 &= A^+
\end{align*}

Thus $O'(2m,p^r)$ is isomorphic to $O^+(2m,p^r)$. The isomorphism $O^+(2m,p^r) \cong O'(2m,p^r)$ is given by $M \in O^+(2m,p^r) \mapsto BMB^{-1}$. $BMB^{-1}$ is in $O'(2m,p^r)$ since 
\begin{align*}
&(BMB^{-1})^T A' (BMB^{-1})\\
&= (B^{-1})^TM^T(B^TA'B)MB^{-1}\\
&= (B^{-1})^T(M^TA^+M)B^{-1}\\
&= (B^{-1})^TA^+B^{-1}\\
&= A'
\end{align*}

 Following the same reasoning as above, we will get that 
$$S'(p,2m) = \{ \begin{pmatrix} A & A\mbf{y}^T & A\mbf{z}^T & AB\\
\mbf{0} & 1 & 0 & -\mbf{y}\\
\mbf{0} & 0 & 1 & \mbf{z}\\
0_{m-1} & \mbf{0} & \mbf{0} & (A^{-1})^T \end{pmatrix} : \mbf{y} \in \F_{p^r}^{m-1}, A \in \Up_{m-1}(\F_{p^r}), B + B^T = \mbf{z}^T\mbf{z}-\mbf{y}^T\mbf{y}\}$$
is a Sylow $p$-subgroup of $O'(2m,p^r)$, and thus it must be isomorphic to $S^+(p,2m)$, a Sylow $p$-subgroup of $O^+(2m,p^r)$. And we can see that $S^-(p,2m)$ is isomorphic to $S'(p,2m)$.

Thus we have that the Sylow $p$-subgroups of $O^+(2m,p^r)$ and $O^-(2m,p^r)$ are isomorphic. 

\end{proof}

\subsection{Lemma \ref{Osyl2}}\label{Ap6}

\begin{lma}[\ref{Osyl2}] Let  
$$S(2,2m) = \{\begin{pmatrix} A & 0_m\\
0_m & (A^{-1})^T\end{pmatrix}\begin{pmatrix} \text{Id}_m & B\\
0_m & \text{Id}_m\end{pmatrix} : A \in \Up_{m}(\F_{2^r}), B \in \asymo(m,2^r)\}.$$ Then $S(2,2m) \in \syl_2(\Omega^\epsilon(2m,2^r))$ for $\epsilon \in \{\pm\}$. \end{lma}

\begin{proof}

Since $\Omega^\epsilon(2m,2^r) \subset O^\epsilon(2m,2^r) \subset Sp(2m,2^r)$, we must have that for $S_1  \in \syl_2(\Omega^\epsilon(2m,2^r)),$ $S_2 \in \syl_2(O^\epsilon(2m,2^r)),$ $S_3 \in \syl_2(Sp(2m,2^r))$, $S_1 \subset S_2 \subset S_3$. It is straightforward to show that for $S_3 \in \syl_2(Sp(2m,2^r)$ for $S_3 = N \rtimes O$ where 
$$N = \{\begin{pmatrix} \text{Id}_m & B\\
0_m & \text{Id}_m\end{pmatrix} : B \in \sym(m,p^r)\} \cong \sym(m,p^r)$$
and
$$O = \{\begin{pmatrix} A & 0_m\\
0_m & (A^{-1})^T\end{pmatrix} : A \in \Up_m(\F_{2^r})\} \cong \Up_m(\F_{2^r}).$$
Note $O$ is a subgroup of both $\Omega^+(2m,2^r)$ and $\Omega^-(2m,2^r)$.  $O$ is isomorphic to $\Up_m(\F_{2^r})$. So $|O| = (2^r)^{m(m-1)/2}$. Let 
$$N' = \{\begin{pmatrix} \text{Id}_m & B\\
0_m & \text{Id}_m\end{pmatrix} : B \in \asymo(m,2^r)\} \subset N$$
Then $N' \cong \asymo(m,2^r)$.  And for $M \in N'$, 
$$M^TA_m^+M = \begin{pmatrix} 0_m & \text{Id}_m\\
0_m & B^T \end{pmatrix}$$
and for $x = (y,z)$, 
$$Q(Mx) = y^Tz + z^TB^Tz$$
And 
\begin{align*}
z^TB^Tz &= \sum_{i,j} B_{i,j}z_iz_j\\
&= \sum_{i < j} 2B_{i,j}z_iz_j + \sum_{i=1}^n B_{i,i}z_i^2 \text{ since } B \in \asymo(m,2^r) \subset \sym(m,2^r)\\
&= 0 \text{ since  we are in characteristic } 2 \text{ and } B_{i,i} = 0, \text{ } \forall i
\end{align*}
Therefore, $Q^+(Mx) = y^Tz = Q^+(x)$ for all $x = (y,z)$. So $N'^+ \subset O^+(2n,p^r)$.  Also, for $M = \begin{pmatrix} \text{Id}_m & B\\
0_n & \text{Id}_m\end{pmatrix} \in N'$, 
\begin{align*} 
M^TA_n^-M &= \begin{pmatrix} \text{Id}_n & 0_n\\
B^T & \text{Id}_m\end{pmatrix} \begin{pmatrix} 0^1_m & \text{Id}_m\\
0_m & 0^d_m\end{pmatrix} \begin{pmatrix} \text{Id}_m & B\\
0_m & \text{Id}_m\end{pmatrix}\\
&= \begin{pmatrix} \text{Id}_m & 0_m\\
B^T & \text{Id}_m\end{pmatrix} \begin{pmatrix} 0^1_m & \text{Id}_m\\
0_m & 0^d_m\end{pmatrix}, \text{ since } B_{m,m} = 0\\
&= \begin{pmatrix} 0^1_m & \text{Id}_m\\
0_m & B^T + 0^d_m\end{pmatrix}
\end{align*}
So for $x = (y,z)$,
\begin{align*}
Q^-(Mx) &= \mbf{y}\mbf{z}^T + y_m^2 + dz_m^2 + \mbf{z}B^T\mbf{z}^T\\
&= \mbf{y}\mbf{z}^T + y_m^2 + dz_m^2  \text{ since } \mbf{z}B^T\mbf{z}^T = 0 \text{ by the work shown above}\\
&= Q^-(x)
\end{align*}
Therefore $N' \subset O^-(2n,p^r)$ as well. And 
$$|N'| = (p^r)^{\sum_{k=1}^{m-1}k} = (p^r)^{m(m-1)/2}.$$
Then consider $N' \rtimes O \subset \Omega^\epsilon(2m,2^r)$ for both $\epsilon = +$ and $\epsilon = -$ (the operation is inherited from $N \rtimes O$). Then we have 
\begin{align*}
|N' \rtimes O| &= |N'| \cdot |O|\\
&= (2^r)^{n(n-1)/2} \cdot (2^r)^{m(m-1)/2}\\
&= 2^{rn(n-1)}
\end{align*}
We learned the following argument from an early draft of \cite{FKW}:\\
\noindent Note that for $M = \begin{pmatrix} A & 0_m\\
0_m & (A^{-1})^T \end{pmatrix} \in O$,

\begin{align*}
\delta_{2m,2^r}^+(M) &= \text{rank}(\text{Id}_{2m} - M) \mod 2\\
&= \text{rank}\begin{pmatrix} \text{Id}_m + A & 0_m\\
0_m & \text{Id}_m + (A^{-1})^T\end{pmatrix} \mod 2\\
&= 2 \text{ rank}(A) \mod 2\\
&= 0
\end{align*}
And for $M = \begin{pmatrix} \text{Id}_m & B\\
0_m & \text{Id}_m\end{pmatrix} \in N'$,

\begin{align*}
\delta_{2m,2^r}^+(M) &= \text{rank}(\text{Id}_{2m} - M) \mod 2\\
&= \text{rank}\begin{pmatrix} 0_m & B\\
0_m & 0_m\end{pmatrix} \mod 2\\
&= \text{rank}(B) \mod 2\\
\end{align*}
And since $B$ is symmetric with $B_{i,i} = 0, \text{ } \forall i$, $B$ determines an alternating symmetric bilinear form, and thus has even rank.

Thus, $\delta_{2m,2^r}^+(M) = 0$ for $M \in N'$ as well.  Hence we have that both $N'$ and $O$ are in $\Omega^+(2m,2^r) = SO^+(2m,2^r) = \ker(\delta_{2m,2^r}^+)$. Therefore, $N' \rtimes O \subset \Omega^+(2n,2^r)$.  And 
$$|N' \rtimes O| = 2^{2m(m-1)} = |\Omega^\epsilon(2m,2^r)|_2$$ Thus we can conclude that for $\epsilon = +, -$,
$$N' \rtimes O \in \syl_2(\Omega^\epsilon(2m,2^r)$$
\end{proof}

\subsection{Lemma \ref{Lem6.18}}\label{Ap7}

For $p \neq 2$, we define 
$$S(p,2m) = \{\begin{pmatrix} A & 0_m\\
0_m & (A^{-1})^T\end{pmatrix}\begin{pmatrix} \text{Id}_m & B\\
0_m & \text{Id}_m\end{pmatrix} : A \in \Up_{m}(\F_{p^r}), B \in \asym(m,p^r)\}.$$
It is easy to show that $S(p,2m)$ is isomorphic to the elements in $\syl_p(\Omega^+(2m,p^r))$ and that
$$S(p,2m) \cong \asym(m,p^r) \rtimes \Up_{m}(\F_{p^r}),$$
where the action is given by $A(B) = ABA^T.$

\begin{lma}[\ref{Lem6.18}] For any prime $p$, $m > 2$, let $S(p,2m) = S(p,2m)$ be defined as above and in Lemma 6.15. Then $$Z(S(p,2m)) = \{\begin{pmatrix} \text{Id}_m & D\\
0_m & \text{Id}_m\end{pmatrix} : D = \begin{pmatrix} 0 & x & \mbf{0}\\
-x & 0 & \mbf{0}\\
\mbf{0} & \mbf{0} & 0_{m-2}\end{pmatrix}\} \cong \F_{p^r}^+ \cong (\Z/p\Z)^r$$
\end{lma}
For the proof, we need the following lemma:

\begin{lemma}\label{ADAO2n} Given $D \in \begin{cases} \asym(m,p^r) &p \neq 2\\
\asym0(m,2^r) &p = 2 \end{cases}$, 
$$AD = D(A^{-1})^T \text{ } \forall A \in \Up_m(\F_{p^r}) \Leftrightarrow D = \begin{pmatrix} 0 & x & \mbf{0}\\
-x & 0 & \mbf{0}\\
\mbf{0} & \mbf{0} & 0_{m-2}\end{pmatrix}.$$ \end{lemma}

\begin{remark} This lemma is true for any $m \geq 2$. \end{remark}
\noindent Granting this lemmma, we can calculate the center:
\begin{proof}

For $p \neq 2$,
\begin{align*}
S(p,2m) &= \{\begin{pmatrix} A & 0_m\\
0_m & (A^{-1})^T\end{pmatrix}\begin{pmatrix} \text{Id}_m & B\\
0_m & \text{Id}_m\end{pmatrix} : A \in \Up_m(\F_{p^r}), B \in \asym(m,p^r)\}\\
&= \{\begin{pmatrix} A & AB\\
0_m & (A^{-1})^T\end{pmatrix} : A \in \Up_m(\F_{p^r}), B \in \asym(m,p^r)\}.
\end{align*}
and 
\begin{align*}
S(2,2m) &= \{\begin{pmatrix} A & 0_m\\
0_m & (A^{-1})^T\end{pmatrix}\begin{pmatrix} \text{Id}_m & B\\
0_m & \text{Id}_m\end{pmatrix} : A \in \Up_m(\F_{p^r}), B \in \asymo(m,2^r)\}\\
&= \{\begin{pmatrix} A & AB\\
0_n & (A^{-1})^T\end{pmatrix} : A \in \Up_m(\F_{p^r}), B \in \asymo(m,2^r)\}.
\end{align*} 
Note that for any $p$, given $$\begin{pmatrix} A & AB\\
0_m & (A^{-1})^T\end{pmatrix}, \begin{pmatrix} C & CD\\
0_m & (C^{-1})^T\end{pmatrix} \in \Omega^+(2m,2^r)$$
we have
{\footnotesize $$\begin{pmatrix} A & AB\\
0_m & (A^{-1})^T\end{pmatrix}^{-1}\begin{pmatrix} C & CD\\
0_m & (C^{-1})^T\end{pmatrix}\begin{pmatrix} A & AB\\
0_m & (A^{-1})^T\end{pmatrix} = \begin{pmatrix} A^{-1}CA & A^{-1}CAB + A^{-1}CD(A^{-1})^T - B((A^{-1}CA)^{-1})^{T}\\
0_m & ((A^{-1}CA)^{-1})^{T}\end{pmatrix}.$$}
So 
$$\begin{pmatrix} C & CD\\
0_m & (C^{-1})^T\end{pmatrix} \in Z(S(p,2m))$$ if and only if 
$$C \in Z(\Up_m(\F_{p^r})) = \{\begin{pmatrix} 1 & 0 & x\\
\mbf{0} & \text{Id}_{m-2} & \mbf{0}\\
0 & \mbf{0} & 1\end{pmatrix}\}$$ and 
$$CD = CB + CA^{-1}D(A^{-1})^T - B(C^{-1})^{T}, \text{ for all } A \in \Up_m(\F_{p^r}),B \in \begin{cases} \asym(m,p^r) &p \neq 2\\
\asym0(m,2^r) &p = 2 \end{cases}.$$

\begin{remark} For the remainder of this proof $p$ can be any prime. (When $p = 2$, the negatives will go away, but the argument is the same.) \end{remark}

\noindent Choosing $A = \text{Id}_m$, we need 
$$CD = CB + CD - B(C^{-1})^T.$$ So we must have 
$$CB = B(C^{-1})^T$$ for all 
$$B \in \begin{cases} \asym(m,p^r) &p \neq 2\\
\asym0(m,2^r) &p = 2 \end{cases}.$$ Write 
$$C = \begin{pmatrix} 1 & \mbf{0} & x\\
\mbf{0} & \text{Id}_{m-2} & \mbf{0}\\
0 & \mbf{0} & 1\end{pmatrix} \in Z(\Up_m(\F_{p^r})).$$
$$(C^{-1})^T = \begin{pmatrix} 1 & \mbf{0} & -x\\
\mbf{0} & \text{Id}_m & \mbf{0}\\
0 & \mbf{0} & 1\end{pmatrix}^{T} = \begin{pmatrix} 1 & \mbf{0} & 0\\
\mbf{0} & \text{Id}_m & \mbf{0}\\
-x & \mbf{0} & 1\end{pmatrix}.$$
Then for 
$$B = (b_{i,j}) \in \begin{cases} \asym(m,p^r) &p \neq 2\\
\asym0(m,2^r) &p = 2\end{cases},$$ we get 

$$CB = \begin{pmatrix} -xb_{1,m} & b_{1,2}-xb_{2,m} & & \cdots & b_{1,m-1} - xb_{m-1,m} & b_{1,m}\\
-b_{1,2} & 0 & b_{2,3} & & \cdots & b_{2,m}\\
\vdots & & \ddots & & & \vdots\\
-b_{1,m-1} & & & \cdots & & b_{m-1,m}\\
-b_{1,m} & & & \cdots & -b_{m-1,m} & 0\end{pmatrix}$$
and 
$$B(C^{-1})^T = \begin{pmatrix} -xb_{1,m} & b_{1,2} & & \cdots & b_{1,m}\\
-b_{1,2} -x b_{2,m} & 0 & b_{2,3} & \cdots & b_{2,m}\\
\vdots & & \ddots & & \vdots\\
-b_{1,m-1} - xb_{m-1,m} & -b_{2,m-1} & \cdots & & b_{m-1,m}\\
-b_{1,m} & -b_{2,m} & \cdots & -b_{m-1,m} & 0\end{pmatrix}$$
So if $m > 2$, we must have $x = 0$, and hence $C = \text{Id}_m$.  

\begin{remark}This is where I need $m > 2$. \end{remark} 

\noindent So the other requirement above becomes 
$$D = A^{-1}D(A^{-1})^T \Leftrightarrow AD = D(A^{-1})^T$$
for all $A \in \Up_m(\F_{p^r})$. Then by Lemma \ref{ADAO2n}, we get that
$$Z(S(p,2m)) = \{\begin{pmatrix} \text{Id}_m & D\\
0_m & \text{Id}_m\end{pmatrix} : D = \begin{pmatrix} 0 & x & \mbf{0}\\
-x & 0 & \mbf{0}\\
\mbf{0} & \mbf{0} & 0_{m-2}\end{pmatrix}\}$$
\end{proof}

\begin{proof}[Proof of Lemma \ref{ADAO2n}]
\text{}

\noindent $\Leftarrow$: This is a straightforward calculation.

\noindent $\Rightarrow$:
Choose $A = \text{Id} + ae_{i_0,j_0}$, where $e_{i_0,j_0}$ denotes the matrix with a $1$ in the $\{i_0,j_0\}th$ component and zeros elsewhere.  Then
$$(AD)_{i_0,k} = D_{i_0,k} + aD_{j_0,k}.$$ 
And for $i \neq i_0$, $(AD)_{i,j} = D_{i,j}$. Also 
$$(D(A^{-1})^T)_{k,i_0} = D_{k,i_0} - aD_{k,j_0}.$$ 
And for $j \neq i_0$, $(AD)_{i,j} = D_{i,j}$.

Thus in order for these to be equal, we must have $D_{k,j_0} = 0$ for all $k \neq i_0$.  Then for any $j_0 > 2$, we can choose $i_0 = j_0 - 1$ and $i_0 = 1$ to get $D_{k,j_0} = 0$ for all $k$.  If $j_0 = 2$, then we only can choose $i_0 = 1$. Thus we get that $D_{k,j_0} = 0$ for all $k \neq 1$.  Thus $D_{k,j} = 0$ for all $k,j$ except $k = 1, j = 2$.  

\end{proof}

\subsection{Section 6.4 Calculation}\label{6.4calc}
The calculation that $H \in L_\mbf{b}$ if and only if $\psi(\mbf{b} \cdot (\mbf{hdh^T} - \mbf{d})) = 1$ for all $\mbf{d} \in (\F_{p^r}^+)^{m(m-1)/2},$
where $\mbf{hdh^T}$ is the vector in $(\F_{p^r}^+)^{m(m-1)/2}$ corresponding to $HDH^T \in \sym(m,p^r)$ under the isomorphsim $\sym(m,p^r) \cong (\F_{p^r}^+)^{m(m-1)/2}$:

\begin{remark} In all of the following, we view $\psi_{(b_j)}$ as a map on $$\begin{cases} \Delta \cong \asym(m,p^r) \cong \F_{p^r}^{m(m-1)/2} &p \neq 2\\
\Delta \cong \asymo(m,2^r) \cong \F_{2^r}^{m(m-1)/2} &p = 2 \end{cases}.$$ So $\psi_{(b_j)}(D,\text{Id}) = \psi_{(b_j)}(D) = \psi(\mbf{b} \cdot \mbf{d})$, where $\mbf{b} = (b_j)$ and $\mbf{d}$ is the vector corresponding to the matrix $D$. \end{remark} 
\noindent The action of $h \in \syl_p(\Omega^+(2m,p^r))$ on $\widehat{\Delta}$ is given by $${}^h\psi(D,\text{Id}_m) = \psi(h^{-1}(D,\text{Id}_m)h).$$
So for $h = (0_m,H^{-1})$, the action on $\psi_{(b_j)}$ is given by $${}^h\psi_{(b_j)}(D,\text{Id}_m) = \psi_{(b_j)}((0_m,H)(D,\text{Id}_m)(0_m,H^{-1})).$$
So $(0_m,H^{-1}) \in L_\mbf{b}$ if and only if

$$\psi_{(b_j)}((0_m,H)(D,\text{Id}_m)(0_m,H^{-1})) = \psi_{(b_j)}(D, \text{Id}_m)$$
for all $$\mbf{d} \in (\F_{p^r}^+)^{m(m-1)/2} \text{ corresponding to }D \in \begin{cases} \asym(m,p^r) &p \neq 2\\
\asymo(m,2^r) &p = 2 \end{cases}.$$
Let $\mbf{hdh^T}$ be the vector corresponding to $HDH^T$. Then since
$$\psi_{(b_j)}((0_m,H)(D,\text{Id}_m)(0_m,H^{-1})) = \psi(\mbf{b} \cdot \mbf{hdh^T}),$$
 and
$$\psi_{(b_j)}(D,\text{Id}_m) = \psi(\mbf{b} \cdot \mbf{d}).$$
we get that  $(0_m,H^{-1}) \in L_\mbf{b}$ if and only if 
$$\psi(\mbf{b} \cdot (\mbf{hdh^T} - \mbf{d})) = 1$$
for all 
$$\mbf{d} \in (\F_{p^r})^{m(m-1)/2} \text{ corresponding to } D \in \begin{cases} \asym(m,p^r) &p \neq 2\\
\asymo(m,2^r) &p = 2 \end{cases}.$$

\subsection{Proposition \ref{Ls1}}\label{Ap8}
\begin{prop}[\ref{Ls1}] For any prime $p$,
$$\min_{\mbf{b} \in  (\F_{p^r}^+)^{m(m-1)/2}, \text{ } b_1 \neq 0} \dim(\theta_{\mbf{b},1}) =  p^{2r(m-2)}.$$
This minimum is achieved when $\mbf{b} = (b,0,\ldots,0)$ with $b \neq 0$.\end{prop}

\begin{proof} 

Write $$H = \begin{pmatrix} 1 & h_{1,2} & h_{1,3} & \cdots & h_{1,n} \\
0 & 1 & h_{2,3} & \cdots & h_{2,m}\\
& & \ddots & & \vdots\\
0 & 0 & \cdots & 1 & h_{m-1,m}\\
0 & 0 & 0 & \cdots & 1\end{pmatrix}, \qquad D = \begin{pmatrix} 0 & d_{1,2} & d_{1,3} & \cdots & d_{1,m}\\
-d_{1,2} & 0 & d_{2,3} & \cdots & d_{2,m}\\
\vdots &  & \ddots &  & \vdots\\
-d_{1,m-1} & -d_{2,m-1} & \cdots & 0 & d_{m-1,m}\\
-d_{1,m} & -d_{2,m} & \cdots & -d_{m-1,m} & 0\end{pmatrix}.$$

We will prove the proposition in two steps:

\textbf{Step 1: Proving that for any } $\mbf{s = (b_i), b_1 \neq 0, |L_\mbf{b}| \leq |\F_{p^r}| \cdot |U_{m-2}(\F_{p^r})| =  p^{2r(m-2)}}.$

\noindent In all the following, in characteristic 2, the negatives will go away, but the argument is the same.

\begin{calc'} For $j_0 > 2$, choose $d_{i,j} = 0$ except for $d_{1,j_0} = -d_{j_0,1}$. \end{calc'} 

\noindent Then $$\mbf{b} \cdot (\mbf{hdh^T} - \mbf{d}) =  \sum_{i=2}^{j_0-1} h_{i,j_0}d_{1,j_0}B_{1,i} = d_{1,j_0}\left(\sum_{i=2}^{j_0-1} h_{i,j_0}B_{1,i}\right)$$
If $\sum_{i=2}^{j_0-1} h_{i,j_0}B_{1,i} \neq 0$, then as we run through all the values for $d_{1,j_0}$, we will get that $\mbf{b} \cdot (\mbf{hdh^T} - \mbf{d})$ runs through all the values of $\F_{p^r}$. And since $\psi$ is non-trivial, this means that $\psi(\mbf{b} \cdot (\mbf{hdh^T} - \mbf{d}))$ cannot always equal $1$.  This is a contradiction. So we must have 
$$\sum_{i=2}^{j_0-1} h_{i,j_0}B_{1,i} = 0$$ for all choices of $j_0 > 2$. Recall that $B_{1,2} = b_1 \neq 0$. So, for all $j_0 > 2$, given $h_{i,j_0}$ for $i > 2$, the above dictates $h_{2,j_0}$: If we know $h_{i,j_0}$ for $i > 1$, then we have
$$\sum_{i=2}^{j_0-1} h_{i,j_0}B_{1,i} = 0 \Rightarrow h_{2,j_0} = \frac{-1}{B_{1,2}} \sum_{i=3}^{j_0-1} h_{i,j_0}B_{1,i}.$$
(In particular, note $h_{2,3} = 0$.) For $3 \leq k \leq n$, if $B_{1,k} \neq 0$, then for all $j_0 > 2$, given $h_{i,j_0}$ for $i \neq 1,k$, the above dictates $h_{k,j_0}$: If we know $h_{i,j_0}$ for $i \neq 1,k$, then we have
$$\sum_{i=2}^{j_0-1} h_{i,j_0}B_{1,i} = 0 \Rightarrow h_{k,j_0} = \frac{-1}{B_{1,k}} \sum_{i=2, i \neq k}^{j_0-1} h_{i,j_0}B_{1,i}.$$

\bigskip

\begin{calc'} Now for $j_0 > 2$, choose $d_{i,j} = 0$ except for $d_{2,j_0} = -d_{j_0,2}$. \end{calc'}

\noindent Then
$$\mbf{b} \cdot (\mbf{hdh^T} - \mbf{d}) = d_{2,j_0}\left(-B_{1,2}h_{1,j_0} + \sum_{i=2}^{j_0} B_{1,i}h_{i,j_0}h_{1,2} + \sum_{i=3}^{j_0-1} B_{2,i}h_{i,j_0}\right)$$
By the same reasoning as before, we must have 
$$-B_{1,2}h_{1,j_0} + \sum_{i=2}^{j_0} B_{1,i}h_{i,j_0}h_{1,2} + \sum_{i=3}^{j_0-1} B_{2,i}h_{i,j_0} = 0$$
for all choices of $j_0 > 2$. Recall that $B_{1,2} = b_1 \neq 0$. So for all $j_0 > 2$, given $h_{i,j_0}$ for $i > 2$, the above dictates $h_{1,j_0}$: If we know $h_{1,2}$ and $h_{i,j_0}$ for $i > 1$, then we have
$$-B_{1,2}h_{1,j_0} + \sum_{i=2}^{j_0} B_{1,i}h_{i,j_0}h_{1,2} + \sum_{k=3}^{j_0-1} B_{2,i}h_{i,j_0} = 0 \Rightarrow h_{1,j_0} = \frac{1}{B_{1,2}} \left(\sum_{i=2}^{j_0} B_{1,i}h_{i,j_0}h_{1,2} + \sum_{i=3}^{j_0-1} B_{2,k}h_{i,j_0}\right)$$

\bigskip

\noindent Thus we can conclude that for all $s = (b_i)$ with $b_1 \neq 0$,
$$|L_\mbf{b}| \leq |\{H : H_{2,j} \text{ fixed }, \forall j > 2, H_{1,j} \text{ fixed }, \forall j > 2\}| = |\F_{p^r}| \cdot |U_{m-2}(\F_{p^r})| = p^{r[(m-2)(m-3)/2 + 1]}.$$

\textbf{Step 2: Exhibiting that the max is achieved when } $\mbf{s = (b, 0, \cdots, 0)}$ \textbf{ with } $\mbf{b \neq 0}.$

\noindent Let $B$ be the matrix corresponding to $s = (b,0,\cdots,0)$. So since the only nonzero entry of $B$ is $B_{1,2} = b$, we have that
 $$\mbf{b} \cdot (\mbf{hdh^T} - \mbf{d}) = b(HDH^T - D)_{1,2} = b\left([\sum_{l = 2}^n [h_{2,l}(\sum_{k=1}^{l} d_{k,l}h_{1,k} - \sum_{k=l+1}^{m-1} d_{l,k}h_{1,k})]] - d_{1,2}\right).$$
By the first calculation above, we have that for $j_0 > 2$,
$$h_{2,j_0} = \frac{-1}{B_{1,2}} \sum_{i=3}^{j_0-1} h_{i,j_0}B_{1,i} = 0.$$
So we have 
$$\mbf{b} \cdot (\mbf{hdh^T} - \mbf{d}) = b\left([\sum_{k=1}^{2} d_{k,2}h_{1,k} - \sum_{k=3}^{m-1} d_{2,k}h_{1,k}] - d_{1,2}\right)$$
By the second calculation above, we have that for $j_0 > 2$,
\begin{align*}
h_{1,j_0} &= \frac{1}{B_{1,2}} \left(\sum_{i=2}^{j_0} B_{1,i}h_{i,j_0}h_{1,2} + \sum_{i=3}^{j_0-1} B_{2,k}h_{i,j_0}\right)\\
&= h_{2,j_0}h_{1,2}\\
&= 0
\end{align*}
So we have
\begin{align*} 
\mbf{b} \cdot (\mbf{hdh^T} - \mbf{d}) &= b\left([\sum_{k=1}^{2} d_{k,2}h_{1,k} - ] - d_{1,2}\right)\\
&= b(d_{1,2}h_{1,1} + d_{2,2}h_{1,2} - d_{1,2})\\
&= 0 \text{ since } h_{1,1} = 0, d_{2,2} = 0
\end{align*}
Thus we have shown that $(0_m,H^{-1}) \in L_\mbf{b}$ if and only if $h_{2,j} = 0, \forall j > 2$ and $h_{1,j} = 0, \forall j > 2$. Therefore, $$L_\mbf{b} = \{(0_m,H^{-1}) : H_{1,j} = 0, \forall j > 2, H_{2,j} = 0, \forall j > 2\}.$$
So $|L_\mbf{b}| = |\F_{p^r}| \cdot |U_{m-2}(\F_{p^r})| = p^{r[(m-2)(m-3)/2 + 1]}$.

\end{proof}

\bibliographystyle{plain}
\bibliography{references}

\end{document}